\documentclass[onecolumn, draftcls]{IEEEtran}

\usepackage{amssymb}
\usepackage{amscd}
\usepackage{amsthm}
\usepackage{amsmath}
\usepackage{latexsym}
\usepackage{graphicx}
\usepackage{delarray}
\usepackage{epic}
\usepackage{float}
\usepackage[dvips]{epsfig}
\usepackage{subfig}

\begin{document}


\newtheorem{theorem}{Theorem}[section]
\newtheorem{definition}[theorem]{Definition}
\newtheorem{lemma}[theorem]{Lemma}
\newtheorem{proposition}[theorem]{Proposition}
\newtheorem{corollary}[theorem]{Corollary}
\newtheorem{example}[theorem]{Example}
\newtheorem{remark}[theorem]{Remark}

\hfuzz5pt 


\newcommand{\gt}{\tilde{g}}
\newcommand{\ft}{\tilde{f}}
\newcommand{\R}{\mathbb{R}}
\newcommand{\Z}{\mathbb{Z}}
\newcommand{\Zt}{\mathbb{Z}^2}
\newcommand{\Zd}{\mathbb{Z}^d}
\newcommand{\Ztd}{0\mathbb{Z}^{2d}}
\newcommand{\Rt}{\R^2}
\newcommand{\Rtd}{\R^{2d}}
\newcommand{\Zp}{Z_p}
\newcommand{\al}{\alpha}
\newcommand{\be}{\beta}
\newcommand{\om}{\omega}
\newcommand{\ga}{\gamma}
\newcommand{\de}{\delta}
\newcommand{\la}{\lambda}
\newcommand{\La}{\Lambda}
\newcommand{\ala}{\la^\circ}
\newcommand{\aLa}{\La^\circ}
\newcommand{\nat}{\natural}
\newcommand{\G}{\mathcal{G}}
\newcommand{\A}{\mathcal{A}}
\newcommand{\B}{\mathcal{B}}
\newcommand{\W}{\mathcal{W}}
\newcommand{\V}{\mathcal{V}}
\newcommand{\M}{\mathcal{M}}
\newcommand{\Sp}{\mathcal{S}}
\newcommand{\Hp}{\mathcal{H}}
\newcommand{\ka}{\kappa}
\newcommand{\cast}{\circledast}
\newcommand{\id}{\mbox{Id}}
\newcommand{\hx}{\hat{x}}
\newcommand{\ie}{{\em i.e., }}
\newcommand{\eg}{{\em e.g., }}

\newcommand{\bV}{\boldsymbol{V}}
\newcommand{\bPhi}{\boldsymbol{\Phi}}
\newcommand{\bI}{\boldsymbol{I}}

\newcommand{\lo}{{\ell^1}}
\newcommand{\Lo}{{L^1}}
\newcommand{\Lt}{{L^2}}
\newcommand{\SO}{{S_0}}
\newcommand{\Lp}{{L^p}}
\newcommand{\Los}{{L^1_s}}
\newcommand{\WCl}{{W(C_0,\ell^1)}}
\newcommand{\lt}{{\ell^2}}

\newcommand{\conv}[2]{{#1}\,\ast\,{#2}}
\newcommand{\twc}[2]{{#1}\,\nat\,{#2}}
\newcommand{\mconv}[2]{{#1}\,\cast\,{#2}}
\newcommand{\set}[2]{\Big\{ \, #1 \, \Big| \, #2 \, \Big\}}
\newcommand{\inner}[2]{\left< #1,#2\right>}
\newcommand{\dotp}[2]{ #1 \, \cdot \, #2}

\newcommand{\Zpd}{\Zp^d}
\newcommand{\I}{\mathcal{I}}
\newcommand{\Zq}{Z_q}
\newcommand{\Zqd}{Zq^d}
\newcommand{\Zak}{\mathcal{Z}_{a}}
\newcommand{\C}{\mathbb{C}}
\newcommand{\F}{\mathcal{F}}

\newcommand{\convL}[2]{{#1}\,\ast_L\,{#2}}
\newcommand{\abs}[1]{\lvert#1\rvert}
\newcommand{\absbig}[1]{\big\lvert#1\big\rvert}
\newcommand{\scp}[1]{\langle#1\rangle}
\newcommand{\norm}[1]{\lVert#1\rVert}
\newcommand{\normbig}[1]{\big\lVert#1\big\rVert}
\newcommand{\normBig}[1]{\Big\lVert#1\Big\rVert}


\title{Non-Invertible Gabor Transforms}

\author{Ewa Matusiak, Tomer Michaeli,~\IEEEmembership{Student~Member,~IEEE}, and Yonina C. Eldar,~\IEEEmembership{Senior~Member,~IEEE}\thanks{This work has been submitted to the IEEE for possible publication.
Copyright may be transferred without notice, after which this version may no longer be accessible.}\thanks{This work
was supported in part by the Israel Science Foundation under Grant no. 1081/07 and by the European Commission in the
framework of the FP7 Network of Excellence in Wireless COMmunications NEWCOM++ (contract no. 216715).}
\thanks{The authors are with the department of electrical engineering, Technion--Israel Institute of
Technology, Haifa, Israel (phone: +972-4-8294682, fax: +972-4-8295757, e-mail: ewa.matusiak@univie.ac.at,
tomermic@tx.technion.ac.il, yonina@ee.technion.ac.il). The third author is also with the University of Vienna.}}


\maketitle


%


\begin{abstract}
Time-frequency analysis, such as the Gabor transform, plays an important role in many signal processing applications.
The redundancy of such representations is often directly related to the computational load of any algorithm operating
in the transform domain. To reduce complexity, it may be desirable to increase the time and frequency sampling
intervals beyond the point where the transform is invertible, at the cost of an inevitable recovery error. In this
paper we initiate the study of recovery procedures for non-invertible Gabor representations. We propose using fixed
analysis and synthesis windows, chosen \eg according to implementation constraints, and to process the Gabor
coefficients prior to synthesis in order to shape the reconstructed signal. We develop three methods to tackle this
problem. The first follows from the consistency requirement, namely that the recovered signal has the same Gabor
representation as the input signal. The second, is based on the minimization of a worst-case error criterion. Last,
we develop a recovery technique based on the assumption that the input signal lies in some subspace of $L_2$. We show
that for each of the criteria, the manipulation of the transform coefficients amounts to a 2D twisted convolution
operation, which we show how to perform using a filter-bank. When the under-sampling factor is an integer, the
processing reduces to standard 2D convolution. We provide simulation results to demonstrate the advantages and
weaknesses of each of the algorithms.
\end{abstract}

\maketitle


\section{Introduction}

Time-frequency analysis has become a popular tool in signal processing. During the past three decades, it has been
successfully used for noise suppression \cite{EM84,CB01}, blind source separation \cite{BA98}, echo cancelation
\cite{LM99,ACV01,AG99}, relative transfer function identification \cite{C04}, beamforming in reverberant environments
\cite{GBW01}, system identification in general \cite{AC07,AC08}, and more. In algorithms based on time-frequency
transforms such as the Gabor representation, there is often a tradeoff between performance and computational
complexity, which can be controlled by adjusting the redundancy of the transform. The latter is determined by the
product of the sampling intervals in time and frequency, which we denote by $a$ and $b$ respectively. Specifically,
as $a$ and $b$ are increased, there are less coefficients per time unit for any given frequency range, and therefore
the amount of computation needed to process the them decreases. This effect is especially notable in adaptive
algorithms, where $a$ directly affects the convergence rate.

The signal processing literature on Gabor-domain algorithms heavily relies on the fundamental requirement that any
signal can be recovered from its coefficients in the transform domain. This requirement leads to the upper bound
$ab\leq1$. However, since the performance-complexity tradeoff is of continuous nature, it seems very restrictive to
limit the discussion to this regime. Specifically, by increasing the sampling intervals beyond this bound we may
further reduce the computational load of any algorithm operating in the Gabor domain. This benefit is obtained at the
expense of additional deterioration in performance. It is important to note that when $ab>1$, an additional source of
performance degradation comes into play, which is the inherent reconstruction error. This is because in this regime,
we can only guarantee perfect reconstruction for signals lying in certain subspaces of $\Lt$, as we show in this
paper, but not for the entire space. Nevertheless, the resulting complexity reduction may be of greater value in some
applications.

In this paper, we explore reconstruction techniques for non-invertible Gabor transforms, namely in which $ab\geq1$.
The fact that in these cases perfect recovery cannot be guaranteed for every signal introduces extra flexibility in
choosing the analysis and synthesis windows of the transform. Specifically, we address the case where both the
analysis and synthesis windows of the transform are specified in advance. They can be chosen according to
implementation considerations, for example finite support windows, or certain multiple-pole windows \cite{TZ96}
admitting an efficient recursive implementation. Our goal, then, is to process the transform coefficients before
reconstruction such that the recovered signal possesses certain desired properties.

To tackle this problem, we borrow several approaches from the field of sampling theory, which has reached a high
degree of maturity in recent years \cite{U00,EM09}. We begin by employing the \emph{consistency} criterion in which
the recovered signal $\ft(t)$ is constructed such that its Gabor transform coincides with that of the original signal
$f(t)$ \cite{UA94}. We then proceed to analyze a \emph{minimax} strategy, where the reconstruction error $\|\ft-f\|$
is minimized for the worst-case input $f(t)$ \cite{ED04}. Both these approaches are prior-free in the sense that they
do not make use of any special properties, or prior knowledge, that might be available on the signal.

A prevalent prior in the sampling literature, is that the signal to be recovered lies in a shift invariant (SI)
subspace of $\Lt$ (see \eg \cite{UAE92,UAE01,AG01,A02,CE04,UB05} and references therein). In fact, signals and images
encountered in many applications can be quite accurately modeled as belonging to some SI space \cite{U00,EM09}, such
as the space of bandlimited functions, the space of polynomial splines and more. Their widespread use can also be
attributed to the link that subspace priors have with stationary stochastic processes \cite{UN05,EU06,RA08,MI08},
which have been shown to constitute realistic priors for the behavior of natural images \cite{RA06}. In this paper,
we generalize the SI-prior used in the sampling community to a richer type of subspaces of $\Lt$, which we term
\emph{shift-and-modulation} (SMI) invariant spaces. The third class of inverse Gabor techniques we consider, then,
makes use of the SMI prior. We show that such a prior can lead to perfect recovery in some cases, given that the
synthesis window is chosen according to the prior. For a fixed synthesis window, which is not matched to the prior,
we show how to achieve the minimal possible reconstruction error for signals in the prior-space.

In each of the three techniques we develop, the processing of the Gabor coefficients amounts to a 2D
twisted-convolution \cite{EMW07} with a certain kernel, which depends on the analysis and synthesis windows. We show
that the twisted-convolution operation can be interpreted in terms of a filter-bank. Furthermore, in the case of
integer under-sampling (\ie when $ab$ is an integer), the resulting process reduces to a standard 2D convolution in
the time-frequency domain. In these cases, we discuss situations in which the 2D convolution kernel is a separable
function of time and frequency. This allows a significant reduction in computation, namely by implementing the 2D
convolution as two successive 1D filtering operations along the time and frequency directions.

%
%
%

The paper is organized as follows. Section~\ref{sec:not} is devoted to notation that will be used throughout the
paper. In Section~\ref{sec:GaborRiesz} we derive conditions on the analysis and synthesis windows such that they
generate Riesz bases for their span, which guarantees that the non-invertible Gabor representation is stable. In
Section \ref{sec:SI} we review sampling and reconstruction schemes in shift-invariant (SI) spaces in order later to
be able to generalize them to the Gabor transform using SMI spaces. Sections~\ref{sec:GaborInt},
\ref{sec:GaborRational} and \ref{sec:subspaceGabor} constitute the central part of the paper, where in the first two
we develop prior-free recovery procedures for Gabor transforms in the integer and rational under-sampling regimes
respectively, and in the last we discuss SMI-prior recoveries. We devote Section~\ref{sec:twc} to describing how
twisted convolution can be realized as a filter-bank and also how to obtain the inverse of a sequence with respect to
twisted convolution. Finally, in Section~\ref{sec:example} we demonstrate the methods we develop for the case in
which both the analysis and synthesis are performed with Gaussian windows.


\section{Notation and Definitions}\label{sec:not}

We will be working throughout the paper with the Hilbert space of complex square integrable functions, denoted by
$\Lt(\R)$, with inner product
\begin{equation}
\inner{f}{g} = \int_{-\infty}^{\infty} f(t) \overline{g(t)}dt \quad \mbox{for all} \quad f,g\in\Lt(\R),
\end{equation}
where $\overline{g(t)}$ denotes the complex conjugate of $g(t)$. The norm, induced by this inner product, is given by
\begin{equation}
\norm{f}^2 = \inner{f}{f}\,.
\end{equation}
The Fourier transform of $f\in\Lt(\R)$ is defined as
\begin{equation}
\F f(\om) = \int_{-\infty}^{\infty} f(t) e^{-2\pi i t \om}\,dt.
\end{equation}
For convenience, we will sometimes write $\hat{f}$ for $\F f$.

In order to ensure stable recovery we focus on subspaces of $\Lt(\R)$ which are generated by frames or Riesz bases. A
collection of elements $\{s_k\}_{k\in\Z}$ is a \emph{frame} for its closed linear span if there exist constants $A>0$
and $B<\infty$ such that
\begin{equation}
A\norm{f}^2 \leq \sum_{k\in\Z} \abs{\inner{f}{s_k}}^2 \leq B\norm{f}^2 \quad \mbox{for all} \quad f\in \overline{\rm
span}\{ s_k \},
\end{equation}
where $\overline{\rm span}$ denotes the closed linear span of a set of vectors. The vectors $\{s_k\}_{k\in\Z}$ form a
\emph{Riesz basis} if there exist $A>0$ and $B<\infty$ such that for all sequences $c\in\lt$
\begin{equation}
A\norm{c}_{\lt}^2 \leq \normBig{\sum_{k\in\Z}c_k s_k }^2 \leq
B\norm{c}_{\lt}^2 \,,
\end{equation}
where $\norm{c}_{\lt}^2 = \sum_{k\in\Z} \abs{c_k}^2$ is the squared $\lt$-norm of $c_k$. A direct consequence of the
lower inequality is that the basis functions are linearly independent, which means that every function in
$\overline{\rm span}\{ s_k \}$ is uniquely specified by its coefficients $c_k$.

The fundamental building blocks of the Gabor representation are the so-called Gabor systems. To define a Gabor
system, let $a>0$ and $b>0$ be such that $ab=q/p$ with $p$ and $q$ relatively prime, and let $T_{ak}$ and $M_{bl}$,
for $k,l\in\Z$, be the translation and modulation operators given by
\begin{eqnarray}
T_{ak} f(t) &=& f(t-ak)\,;\\
M_{bl} f(t) &=& e^{2\pi i bl t} f(t)\,.
\end{eqnarray}
For $s\in\Lt(\R)$, the \emph{Gabor system} $\G(s,a,b)$ is a collection $\{M_{bl} T_{ak} s(t)\,;\, (k,l) \in \Zt \}$.
The composition
\begin{equation}
M_{bl}T_{ak}f(t) = e^{2 \pi i bl t} f(t-ak),
\end{equation}
which is a unitary operator, is called a
\emph{time-frequency shift operator}. Many technical details in time-frequency analysis are linked to the commutation
law of the translation and modulation operators, namely
\begin{equation}
M_{bl}T_{ak} = e^{2 \pi i abkl}T_{ak}M_{bl}.
\end{equation}
When $p=1$, the time-frequency shift operators commute, i.e. $M_{bl}\,T_{ak} = T_{ak}\,M_{bl}$, because $e^{2 \pi i
abkl}=1$ for all $k,l\in\Z$. One consequence of the commutation rule, which we will use in our exposition, is the
relation
\begin{equation}
\inner{f}{M_{bl-bn}T_{ak-am}f} = e^{2 \pi i ab(l-n)m} \inner{M_{bn}T_{am}f}{M_{bl}T_{ak}f}.
\end{equation}
When $p=1$ this becomes $\inner{f}{M_{bl-bn}T_{ak-am}f} = \inner{M_{bn}T_{am}f}{M_{bl}T_{ak}f}$.

For $s\in\Lt(\R)$, the collection $\G(s,a,b)$ is a Riesz basis for
its closed linear span if there exist bounds $A>0$ and $B<\infty$
such that
\begin{equation}
A\norm{c}_{\lt}^2 \leq \normBig{\sum_{k,l\in\Z} c_{k,l}M_{bl}T_{ak}s }^2 \leq B\norm{c}_{\lt}^2 \qquad c \in \lt,
\end{equation}
and is a frame when
\begin{equation}
A\norm{f}^2 \leq \sum_{k,l\in\Z} \abs{\inner{f}{M_{bl}T_{ak}s}}^2 \leq B\norm{f}^2 \quad \mbox{for all} \quad f\in
\overline{\rm span}\{M_{bl}T_{ak}s \}.
\end{equation}
A necessary condition for $\G(s,a,b)$ to constitute a frame for $\Lt(\R)$ is that $ab\leq 1$, \cite{G01}. Moreover,
if $\G(s,a,b)$ is a frame, then it is a Riesz basis for $\Lt(\R)$ if and only if $ab=1$ \cite{G01}. In this paper we
focus on the regime $ab\geq 1$, where $\G(s,a,b)$ does not necessarily span $\Lt(\R)$.

With a Gabor system $\G(s,a,b)$ we associate a \emph{synthesis operator} (or \emph{reconstruction operator}) $S:
\lt(\Zt) \rightarrow \Lt(\R)$, defined as
\begin{equation}
Sc = \sum_{k,l\in \Z} c_{k,l} M_{bl} T_{ak} s(t) \quad \mbox{for every} \quad c\in\lt(\Zt).
\end{equation}
The conjugate $S^{\ast}:\Lt(\R) \rightarrow \lt(\Zt)$ of $S$ is called the \emph{analysis operator} (or
\emph{sampling operator}), and is given by
\begin{equation}
S^{\ast}f = \{\inner{f}{M_{bl} T_{ak} s}\} \quad \mbox{for every} \quad f\in\Lt(\R).
\end{equation}


\section{Stable Gabor representations}\label{sec:GaborRiesz}

The Gabor representation of a signal $f(t)$ comprises the set of coefficients $\{c_{k,l}\}_{k,l\in\Z}$ obtained by
inner products with the elements of some Gabor system $\G(s,a,b)$ \cite{G01}:
\begin{equation}
c_{k,l} = \inner{f}{M_{bl} T_{ak} s}. \label{eq:analysis}
\end{equation}
This process can be represented as an analysis filter-bank, as shown in Fig.~\ref{fig:GaborSampRec}(a). Consequently,
$s(t)$ is referred to as the analysis window of the transform. If $\G(s,a,b)$ constitutes a frame or Riesz basis for
$\Lt(\R)$, then there exists a function $v(t)\in\Lt(\R)$ such that any $f(t)\in\Lt(\R)$ can be reconstructed from the
coefficients $\{c_{k,l}\}_{k,l\in\Z}$ using the formula
\begin{equation}
f(t) = \sum_{k,l\in\Z}c_{k,l}M_{bl}T_{ak} v(t). \label{eq:synthesis}
\end{equation}
The Gabor system $\G(v,a,b)$ is the dual frame (Riesz basis) to $\G(s,a,b)$. Consequently, the synthesis window
$v(t)$ is referred to as the dual of $s(t)$. The recovery process can be represented as a synthesis filter-bank, as
shown in Fig.~\ref{fig:GaborSampRec}(b).

Generally, there is more than one dual window $v(t)$. It can be shown that any function $v(t)$ satisfying
$\inner{v}{M_{l/a}T_{k/b}s} = \delta_{k}\delta_{l}$ is a dual window. The canonical dual window is given by
$v=Q^{-1}s$, where $Q$ is the frame operator associated to $s(t)$, which is defined by
$Qf=\sum_{k,l\in\Z}\inner{f}{M_{bl}T_{ak}s}M_{bl}T_{ak}s(t)$. There are several ways of finding an inverse of $Q$, namely
by employing the Janssen representation of $Q$, through the Zak transform method or iteratively using one of several
available efficient algorithms \cite{G01}.

\begin{figure*}\centering
\subfloat[Analysis filter bank]{\includegraphics[scale=0.725, trim=0cm 0cm -1cm 0cm ]{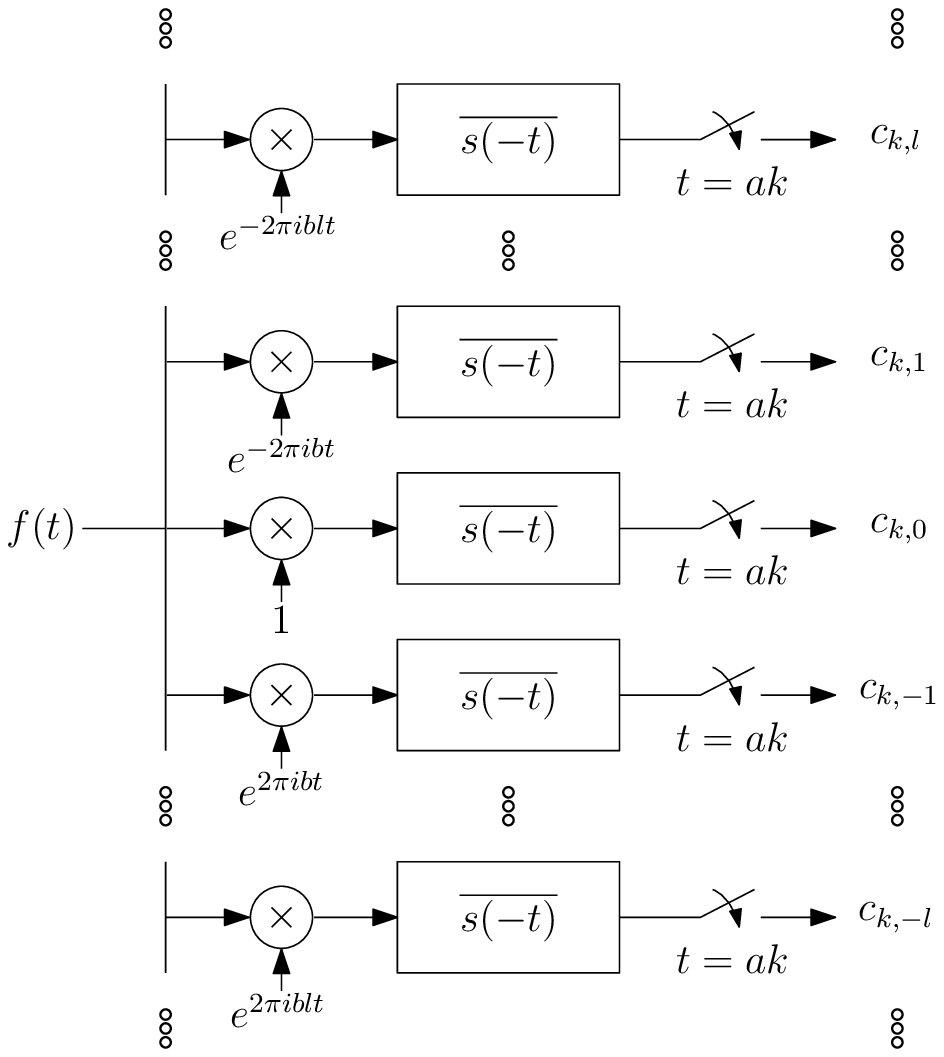}}
\subfloat[Synthesis filter bank]{\includegraphics[scale=0.725, trim=-1cm 0cm 0cm 0cm ]{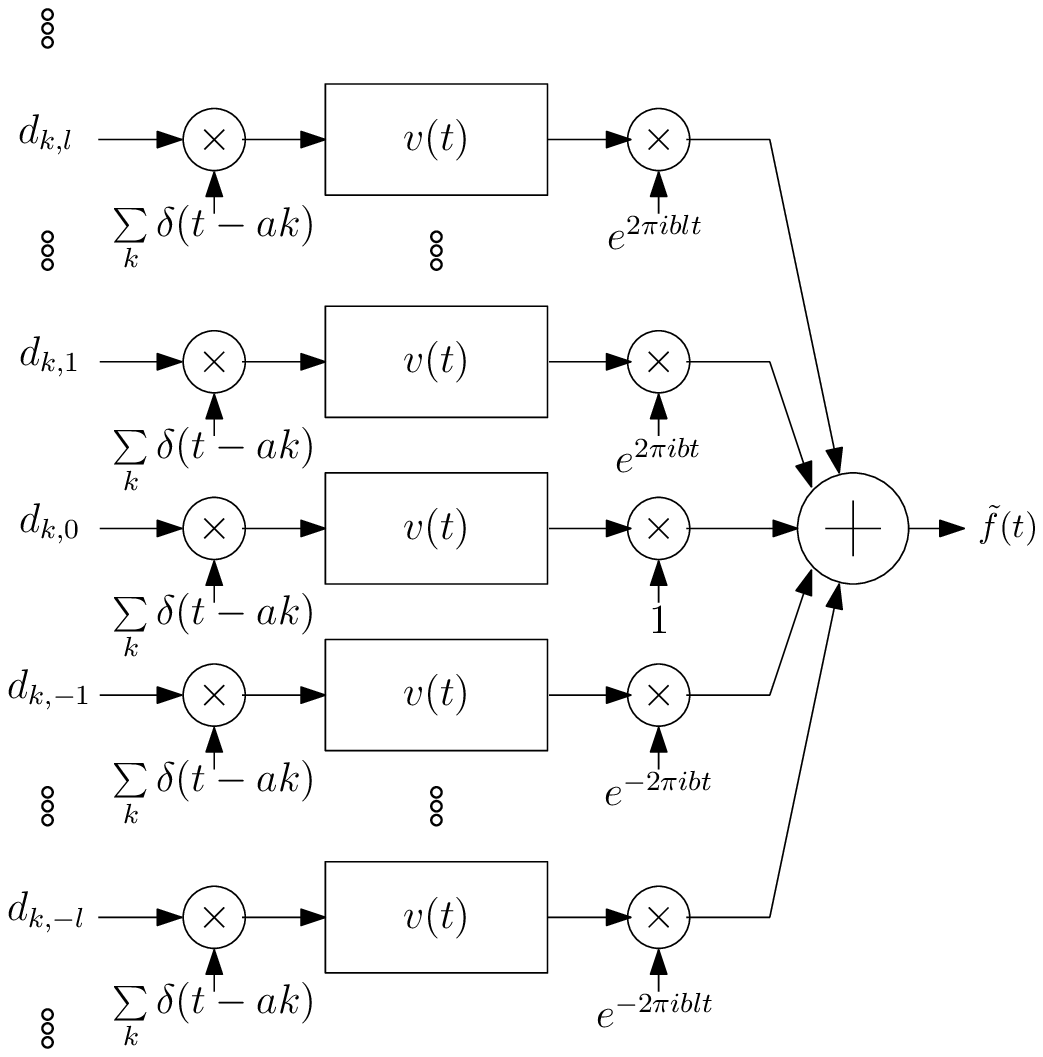}}\\
\caption{Filter-bank representation of the Gabor transform~(a) and of the inverse Gabor transform~(b).}
\label{fig:GaborSampRec}
\end{figure*}

In this paper, we are interested in Gabor systems that do not necessarily span $\Lt(\R)$ but rather only a (Gabor)
subspace. A \emph{Gabor space} is the set $\V$ of all signals that can be expressed in the form (\ref{eq:synthesis})
with some norm-bounded sequence $c_{k,l}$. Since perfect recovery cannot be guaranteed for every signal in $\Lt(\R)$
in these situations, we have the freedom of choosing the analysis and synthesis windows according to implementation
constraints. However, in order for the analysis and synthesis processes to be stable, we would still like to assure
that the systems $\G(s,a,b)$ and $\G(v,a,b)$ form frames or Riesz bases for their span. In this section, we give
several equivalent characterizations of windows $v(t)$ and sampling intervals $a$ and $b$ such that the Gabor system
$\G(v,a,b)$ forms a Riesz basis.

For tractability, we assume throughout the paper that $a$ and $b$ are positive constants such that $ab=q/p$, where
$p$ and $q$ are relatively prime. Moreover, we will consider only Gabor spaces whose generators come from the
so-called \emph{Feichtinger algebra} $\SO$, which is defined by
\begin{equation}
\SO = \set{f \in \Lt(\R)}{\norm{f}_{\SO} := \int \abs{\inner{f}{M_{\om} T_x \psi}} \, dx\, d\om < \infty},
\end{equation}
where $\psi(t)$ is a Gussian window. An important property of $\SO$ is that if $v(t)$ and $s(t)$ are elements from
$\SO$ then $\{\inner{v}{M_{bl} T_{ak} s}\}_{k,l\in\Z}$ is an $\lo(\Zt)$ sequence. Examples of functions in $\SO$ are
the Gaussian and B-splines of any order. The Feichtinger algebra is an extremely useful space of ``good'' window
functions in the sense of time-frequency localization. Rigorous descriptions of $\SO$ can be found in \cite{G01} and
references therein.

The first characterization of Gabor Riesz bases we consider is
stated directly in terms of their generator $v(t)$. It is a simple
corollary of a result on Gabor frames for $\Lt(\R)$, see \cite{G01}.
\begin{proposition}\label{prop:RieszV}
Let $v(t)\in\SO$ and $ab=q/p$ with $p$ and $q$ relatively prime. The collection $\G(v,a,b)$ is a Riesz basis for its
closed linear span if and only if there exist constants $A>0$ and $B<\infty$ such that
\begin{equation}\label{eq:riesz condition}
A\bI_p \leq \bV(\om,x) \leq B\bI_p \quad \text{for almost all } (\om,x)\in\R^2,
\end{equation}
where $\bI_p$ is the $p\times p$ identity matrix and $\bV(\om,x)$ is a $p\times p$ matrix-valued function with
entries given by
\begin{equation}\label{eq:V-hat}
\bV_{r,s}(\om,x) = \frac{1}{b} \sum_{k,l\in\Z} v\left(x - ar - \frac{qk+l}{b}\right) \overline{v\left(x - as -
\frac{l}{b}\right)} e^{-2\pi i ak \om}, \quad r,s=0,\ldots,p-1.
\end{equation}
$\G(v,a,b)$ is an orthonormal basis if $A=B=1$.
\end{proposition}
\begin{proof}
By the Ron-Shen duality principle \cite{RS97}, $\G(v,a,b)$ is a Riesz basis (orthonormal basis) for its closed linear
span if and only if the system $\G(v,1/b,1/a)$ is a frame (Parseval frame) for $\Lt(\R)$. The latter is a frame for
$\Lt(\R)$ if and only if there exist constants $A>0$ and $B<\infty$ such that the so-called frame operator $S_{vv}$,
defined as $S_{vv}f(t) = \sum_{k,l\in\Z} \inner{f}{M_{l/a}T_{k/b}v} M_{l/a}T_{k/b}v(t)$ satisfies
\begin{equation}\label{eq:S_vv}
\frac{A}{b}I \leq S_{vv} \leq \frac{B}{b}I,
\end{equation}
where $I$ is the identity operator on $\Lt(\R)$. This means that $S_{vv}$ is bounded and invertible on $\Lt(\R)$. It
was shown in \cite{G01} that, since $1/(ab) = p/q$, the operator $S_{vv}$ satisfies (\ref{eq:S_vv}) if and only if
(\ref{eq:riesz condition}) is satisfied, which completes the proof.
\end{proof}
Note that $\om$ is a frequency variable associated with the discrete-time variable $k$, and similarly $x$ is a time
variable associated with the discrete frequency index $l$. Another valuable observation is that that $\bV(\om,x)$ is
a $(1/a,1/b)$-periodic function. Furthermore, it has been shown in \cite{G01} that $\bV_{r,s}(\om,x)$ is continuous.
Therefore, the lower bound in (\ref{eq:riesz condition}) can be replaced by the requirement that $\det(\bV(\om,x))>0$
for all $(\om,x)\in [0,1/a) \times [0,1/b)$.

The next characterization we consider is in terms of the twisted convolution operator. Specifically, the Riesz basis
condition implies that $\G(v,a,b)$ is a Riesz basis for its closed linear span if and only if there exist constants
$A>0$ and $B<\infty$ such that
\begin{equation}
A\inner{c}{c} \leq \inner{\twc{r_{vv}}{c}}{c} \leq B\inner{c}{c} \quad \mbox{for all } c\in\lt(\Zt),
\end{equation}
where the 2D cross-correlation sequence $r_{vv}[k,l]$ is defined as
\begin{equation}
r_{vv}[k,l] = \inner{v}{M_{bl}T_{ak}v}.
\end{equation}
The operation $\nat$ represents the twisted convolution defined by
\begin{equation}\label{eq:twc}
(\twc{d}{c})[k,l] = \sum_{m,n\in\Z} d_{m,n} c_{k-m,l-n} e^{-2\pi i ab(l-n)m}.
\end{equation}
When $p=1$, twisted convolution becomes standard convolution, because the exponential term in (\ref{eq:twc}) equals
$1$ for all $m,n,l\in\Z$. Therefore, $v(t)$ generates a Riesz basis if and only if the twisted convolution (standard
convolution when $p=1$) operator with kernel $r_{vv}[k,l]$ is bounded and invertible. Invertibility of this operator
translates to the invertibility of the sequence $r_{vv}[k,l]$ with respect to $\nat$ ($\ast$ respectively).
Proposition~\ref{prop:RieszV} states then, that this twisted convolution operator is bounded and invertible if and
only if the matrix-valued function $\bV(\om,x)$ is bounded and invertible for all $\om$ and $x$. Explicitly finding
the inverse of a sequence with respect to twisted convolution is not a trivial task. We will address the problem in
Section~\ref{sec:twc}.

Our last representation follows from restating Proposition \ref{prop:RieszV} using a different, but equivalent,
matrix-valued function that involves the cross-correlation sequence $r_{vv}[k,l]$ defined earlier. This new
representation was first introduced in \cite{WES05} to characterize the invertibility of general Gabor frame
operators.
\begin{proposition}{\cite{WES05}}\label{prop:RieszPhi}
The matrix-valued function $\bV(\om,x)$ of (\ref{eq:V-hat}) coincides with the matrix-valued function $\bPhi(\om,x)$
whose entries are given by
\begin{equation}\label{eq:Phi-def}
\bPhi_{r,s}(\om,x) = \sum_{k,l\in\Z} r_{vv}[s-r+pk,l] e^{-2\pi i ablr} e^{-2\pi i (blx + ak\om)},
\end{equation}
and therefore $\G(v,a,b)$ is a Riesz basis for its closed linear span if and only if there exist constants $A>0$ and
$B<\infty$ such that
\begin{equation}\label{eq:Phi-bounds}
A\bI_p \leq \bPhi(\om,x) \leq B\bI_p \quad \text{for almost all } (\om,x)\in\R^2.
\end{equation}
\end{proposition}
In the integer under-sampling case $p=1$, $\bPhi(\om,x)$ of (\ref{eq:Phi-def}) reduces to the scalar function
\begin{equation}\label{eq:Phi-boundsPeq1}
\Phi(\om,x) = \sum_{k,l\in\Z} r_{vv}[k,l] e^{-2\pi i (blx + ak\om)} = (\F r_{vv})(\om,x),
\end{equation}
where $\F r_{vv}$ is the 2D discrete-time Fourier transform (DTFT) of $r_{vv}[k,l]$. Therefore, in this case
condition (\ref{eq:Phi-bounds}) reduces to
\begin{equation}\label{eq:Phi-condition}
A \leq (\F r_{vv})(\om,x) \leq B \quad \text{for almost all } (\om,x)\in\R^2
\end{equation}
for some $A>0$ and $B<\infty$.

The $\bPhi$-characterization is of particular interest in our context as it can be used to investigate any twisted
convolution operation with a sequence $h\in\lo(\Zt)$. Indeed, it was shown in \cite{WMES07} that such an operation is
invertible if and only if the matrix-valued function
\begin{equation}\label{eq:phiHdef}
\bPhi^h_{r,s}(\om,x) = \sum_{k,l\in\Z} h_{s-r+pk,l} e^{-2\pi iabrl} e^{-2\pi i (blx + ak\om)}
\end{equation}
is invertible for all $\om$ and $x$. In fact, in some sense the function $\bPhi(\om,x)$ is to twisted convolution
what the DTFT is for convolution. Specifically, we show in Appendix~\ref{sec:appPhiV} that for two sequences
$c_{k,l}$ and $d_{k,l}$ having $\bPhi$-representations $\bPhi^d(\om,x)$ and $\bPhi^c(\om,x)$ respectively, the
matrix-valued function $\bPhi^{(\twc{c}{d})}(\om,x)$ associated to the twisted convolution $\twc{c}{d}$, can be
expressed as
\begin{equation}
\bPhi^{(\twc{c}{d})}(\om,x) = \bPhi^d(\om,x)\bPhi^c(\om,x).
\end{equation}

We conclude this section with the observation that having a Riesz basis for a Gabor space $\V$, it is possible to
construct many others using equivalent generating functions.
\begin{proposition}\label{prop:equivalent basis}
Let $\G(v,a,b)$ be a Riesz basis for its closed linear span $\V$ and $ab=q/p$ with $p$ and $q$ relatively prime. Let
\begin{equation}
w(t)=\sum_{k,l\in\Z} h_{k,l} M_{bl}T_{ak}v(t),
\end{equation}
where $\{h_{k,l}\}$ is a sequence of weights. Then $\G(w,a,b)$ is an equivalent Riesz basis for $\V$ if and only if
there exist constants $A>0$ and $B<\infty$ such that the $(p\times p)$-matrix-valued function $\bPhi^h(\om,x)$ of
(\ref{eq:phiHdef}) satisfies
\begin{equation}\label{eq:PhiPhi*}
A\bI_p \leq \bPhi^h(\om,x)\bPhi^h(\om,x)^{H} \leq B\bI_p \quad \mbox{for almost all } (\om,x)\in \Rt,
\end{equation}
where $\bPhi^h(\om,x)^{H}$ denotes the conjugate transpose of $\bPhi^h(\om,x)$.
\end{proposition}
\begin{proof}
See Appendix~\ref{sec:appProofEquivBasis}.
\end{proof}
In the case of integer under-sampling (\ie when $p=1$), $\bPhi^h(\om,x)$ becomes a scalar function, which is simply
the 2D DTFT of $h_{k,l}$. In this setting, condition (\ref{eq:PhiPhi*}) becomes
\begin{equation}
A\leq \abs{\Phi^h(\om,x)}^2 \leq B \quad \mbox{for almost all } (\om,x)\in \Rt.
\end{equation}

\section{Sampling and Reconstruction in Shift-Invariant
Spaces}\label{sec:SI}

To address the recovery of a function $f(t)$ from its non-invertible Gabor transform, we will harness several
strategies which were initially developed in the context of sampling theory. Specifically, the last two decades have
witnessed a substantial amount of research devoted to the problem of recovering a signal $f(t)$ from the equidistant
point-wise samples of its filtered version, using a predefined reconstruction filter \cite{U00,EM09,TE10}. As can be
seen in Fig.~\ref{fig:SISampRec}, the sampling stage in this setting, corresponds to the central branch in the
analysis filter-bank of the Gabor transform shown in Fig.~\ref{fig:GaborSampRec}(a). Thus, the time-frequency plane
is sampled in this scenario only on the lattice $\{(ak,0)\}_{k\in\Z}$. Similarly, the reconstruction process of
Fig.~\ref{fig:SISampRec} can be identified with the central branch of the synthesis filter-bank of
Fig.~\ref{fig:GaborSampRec}(b).

\begin{figure*}\centering
\begin{tabular}{cc}
\subfloat[Sampling]{\includegraphics[scale=0.75, trim=0cm -1.51cm 0cm 0cm ]{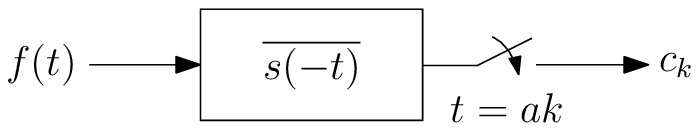}}&
\subfloat[Reconstruction]{\includegraphics[scale=0.75, trim=0cm 0cm 0cm 0cm ]{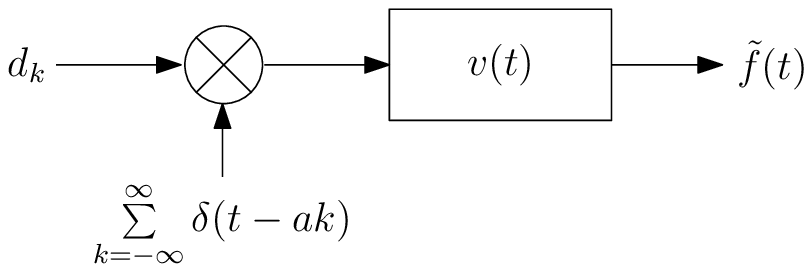}}
\end{tabular}
\caption{Sampling (a) and reconstruction (b) with given filters.} \label{fig:SISampRec}
\end{figure*}

The main goal in this setting is to produce a set of expansion coefficients $\{d_k\}$ by processing the samples
$\{c_k\}$, such that the recovered signal $\tilde{f}(t)$ possesses certain desired properties. In this section we
briefly review three methods for tackling this problem, each based on a different design criterion. For more detailed
explanations and a review of other methods, we refer the reader to \cite{UA94,EW05,ED04,EM09,TE10}. In the following
sections, we will extend these results to the Gabor scenario.

For simplicity, we assume here that $a=1$. The reconstruction process of Fig.~\ref{fig:SISampRec} can be written in
operator notation as $\tilde{f}=Vd$, where $V:\lt\rightarrow\Lt(\R)$ is the synthesis operator associated with the
functions $\{v(t-k)\}_{k\in\Z}$, defined as
\begin{equation}
Vd = \sum_{k\in\Z}d_{k}v(t-k) = \sum_{k\in\Z}d_{k}T_k v(t) \quad \mbox{for every } d\in\lt(\Z).
\end{equation}
Similarly, since $c_k=\inner{f(t)}{s(t-k)}$, the sequence of samples $\{c_k\}$ are obtained by applying the synthesis
operator $S^*$, which is the conjugate of the analysis operator $S$ associated with the functions
$\{s(t-k)\}_{k\in\Z}$:
\begin{equation}
S^{\ast}f = \{\inner{f(t)}{s(t-k)}\} = \{\inner{f}{T_k s}\} \quad \mbox{for every } f\in\Lt(\R).
\end{equation}
We will refer to $\Sp=\overline{\rm span}\{v(t-k)\}$ and $\V=\overline{\rm span}\{v(t-k)\}$ as the sampling and
reconstruction spaces respectively. Spaces of this type are called shift-invariant (SI).

As in the Gabor transform, we will focus on cases where the sets of functions $\{s(t-n)\}$ and $\{v(t-n)\}$
constitute Riesz bases for their span. Then, both the sampling and reconstruction are stable procedures. It is well
known \cite{A96} that the functions $\{v(t-n)\}$ form a Riesz basis for their span $\V$ if and only if there exist
constants $A>0$ and $B<\infty$ such that
\begin{equation}
A \leq \phi_{VV}(\om) \leq B \quad \text{for almost all }\om\in\R,
\end{equation}
where
\begin{equation}\label{eq:phi}
\phi_{VV}(\om) = \frac{1}{2\pi}\sum_{k\in\Z} \left|\hat{v}(\om-k)\right|^2
\end{equation}
is the DTFT of the cross-correlation sequence
\begin{equation}\label{eq:cross-corelation-v}
r_{vv}[n] = \inner{v(t)}{v(t-n)} = \left(\conv{v(t)}{\overline{v(-t)}}\right)(n),
\end{equation}
and $\hat{v}(\om)$ is the Fourier transform of $v(t)$. In other words, $\{v(t-n)\}$ is a Riesz basis if and only if
the sequence $r_{vv}[n]$ is bounded and invertible in the convolution algebra $\lo(\Z,\ast)$. In particular, the
functions $\{v(t-n)\}$ form an orthonormal basis if and only if $A=B=1$. Notice the analogy with condition
(\ref{eq:Phi-bounds}) (and (\ref{eq:Phi-condition}) in the case $p=1$), which was developed for Gabor systems.

\subsection{Consistent reconstruction}
\label{sec:consSI}

Perhaps the most intuitive demand from the recovered signal $\tilde{f}(t)$ is that it would produce the same sequence
of samples $\{c_k\}$ were it re-injected to the sampling device of figure \ref{fig:SISampRec}(a), namely
\begin{equation}
\inner{\tilde{f}(t)}{s(t-k)} = c_k = \inner{f(t)}{s(t-k)}
\end{equation}
for all $k\in\Z$. This \emph{consistency} requirement was first introduced in \cite{UA94} in the context of sampling
in SI spaces and then generalized to arbitrary spaces in \cite{E02,EW05}. There, it was shown that consistent
reconstruction is possible under the direct-sum condition $\Sp^{\perp} \oplus \V = \Lt(\R)$, where $\oplus$ denotes a
sum of two subspaces that intersect only at the zero vector. This means that $\Sp^{\perp}$ and $\V$ are disjoint and
together span the space $\Lt(\R)$.

In the SI setting, the direct-sum condition translates into the simple requirement that \cite{CE04}
\begin{equation}\label{eq:phi condition}
\abs{\phi_{SV}(\om)} > A, \quad \text{for almost all }\om\in\R
\end{equation}
for some positive constant $A$, where
\begin{equation}\label{eq:CrossCorSI}
\phi_{SV}(\om) = \frac{1}{2\pi}\sum_{k\in\Z} \overline{\hat{s}(\om-k)} \hat{v}(\om-k)
\end{equation}
is the DTFT of the cross-correlation sequence $r_{sv}[n] = \inner{s(t)}{v(t-n)}$. Under this condition,
reconstruction can be obtained by convolving the sample sequence $\{c_k\}$ with the filter $h_{\text{con}}$, whose
DTFT is given by \cite{UA94,E03,E04}
\begin{equation}
H_{\text{con}}(\om) = \frac{1}{\phi_{SV}(\om)},
\end{equation}
to obtain the sequence of expansion coefficients $\{d_k\}$.

If $\Sp$ and $\V$ are two arbitrary subspaces of $\Lt(\R)$ satisfying $\Sp^{\perp} \oplus \V = \Lt(\R)$ (namely not
necessarily SI spaces), spanned by the functions $\{s_n(t)\}$ and $\{v_n(t)\}$ respectively, then the sequence of
expansion coefficients $d$ can be obtained by applying the the operator
\begin{equation}
H_{\text{con}}=(S^*V)^{-1}
\end{equation}
on the sequence of samples $c$, where $S$ and $V$ are the synthesis operators associated with $\{s_n(t)\}$ and
$\{v_n(t)\}$ respectively. The direct-sum requirement guarantees that $S^*V:\lt\rightarrow\lt$ is continuously
invertible. In the next sections, we will use this latter characterization to develop a consistent reconstruction
procedure for non-invertible Gabor transforms.

\subsection{Minimax regret reconstruction}
\label{sec:minimaxSI}

A drawback of the consistency approach is that the fact that $f(t)$ and $\tilde{f}(t)$ yield the same samples does
not necessarily imply that $\tilde{f}(t)$ is close to $f(t)$. Indeed, for a signal $f(t)$ not in $\V$, the norm of
the resulting reconstruction error $\tilde{f}(t)-f(t)$ can be arbitrarily large, if $\Sp$ is nearly orthogonal to
$\V$.

To directly control the reconstruction error, it is important to notice that $\tilde{f}(t)$ is restricted to lie in
$\V$ by construction. Therefore, the best possible recovery is the orthogonal projection of $f(t)$ onto $\V$, namely
$\tilde{f}=P_{\V}f$, a fact that follows from the projection theorem \cite{B99}. This solution cannot be generated in
general, because we do not know $f(t)$ but rather only the sequence of samples $\{c_k\}$ it produced. The difference
between the squared-norm error of any recovery $\tilde{f}(t)$ and the smallest possible error, which is
$\|f-P_{\V}f\|^2=\|P_{\V^\perp} f\|^2$, is called the {\em regret} \cite{ENB03}. The regret depends in general on
$f(t)$ and therefore generally cannot be minimized uniformly for all $f(t)$. Instead, the authors in \cite{ED04}
proposed minimizing the worst-case regret over all bounded-norm signals $f(t)$ that are consistent with the given
samples, which results in the problem
\begin{equation}\label{eq:regret2}
\min_{\tilde{f} \in \V} \max_{f\in\B} \|\tilde{f}-f\|^2 -\|P_{\V^\perp} f\|^2,
\end{equation}
where $\B=\{f:S^*f=c,\|f\|\leq L\}$ is the set of feasible signals.

It was shown in \cite{ED04} that the minimax-regret reconstruction can be obtained by filtering the samples $c_k$
with the filter $h_{\text{mx}}$ whose DTFT is given by
\begin{equation}
H_{\text{mx}}(\om) = \frac{\phi_{VS}(\om)}{\phi_{SS}(\om) \phi_{VV}(\om)},
\end{equation}
where $\phi_{VS}(\om)$, $\phi_{SS}(\om)$ and $\phi_{VV}(\om)$ are as in (\ref{eq:CrossCorSI}) with the corresponding
substitution of the generators $v(t)$ and $s(t)$.  Note that the solution is independent of the bound $L$ appearing
in the definition of $\B$.

If the sampling and reconstruction functions form Riesz bases for arbitrary spaces $\Sp$ and $\V$ (not necessarily
SI), then the sequence of expansion coefficients $d$ can be obtained by applying the operator
\begin{equation}\label{eq:generalMinimax}
H_{\text{mx}}=(V^*V)^{-1}S^*V(S^*S)^{-1}
\end{equation}
on the sequence of samples $c$. The operators $V^*V$ and $S^*S$ are guaranteed to be continuously invertible due to
the Riesz basis assumption. This more general characterization will be used in the next sections to develop a
minimax-regret recovery method for non-invertible Gabor transforms.

\subsection{Subspace-prior reconstruction}
\label{sec:subspaceSI}

The consistent reconstruction approach leads to perfect recovery for input signals that lie in the reconstruction
space $\V$ \cite{UA94}. The minimax-regret method, on the other hand, leads to the best possible approximation
$\tilde{f}=P_{\V}f$ for signals $f(t)$ lying in the sampling space $\Sp$ \cite{ED04}. Therefore, the two methods can
be thought of as emerging from the prior that $f(t)$ lies in a certain subspace $\W$ of $\Lt(\R)$, where $\W=\V$ in
the consistent strategy and $\W=\Sp$ in the minimax-regret approach. In practice, though, it is often desirable to
choose the sampling and reconstruction spaces according to implementation constraints and not to reflect our prior
knowledge on the typical signals entering our sampling device. Thus, commonly neither constitutes a subspace prior
$\W$, which is good in the sense that $\|f-P_{\W}f\|$ is small for most signals in our application.

A generalization of these two methods results by assuming that $f\in\W$ where $\W=\overline{\rm span}\{w(t-k)\}$ for
a generator $w(t)$, which may be different than $s(t)$ and $v(t)$. If the subspace $\W$ satisfies the direct-sum
condition $\Sp^{\perp} \oplus \W = \Lt(\R)$, then the solution $\tilde{f} = P_{\V}f$ can be generated  by filtering
the sample sequence $c_k$ with \cite{ED04}
\begin{equation}
H_{\text{sub}}(\om) = \frac{\phi_{VW}(\om)}{\phi_{SW}(\om) \phi_{VV}(\om)},
\end{equation}
where $\phi_{VW}(\om)$, $\phi_{SW}(\om)$ and $\phi_{VV}(\om)$ are as in (\ref{eq:CrossCorSI}) with the appropriate
substitution of $v(t)$, $s(t)$, and $w(t)$.

For arbitrary sampling, reconstruction and prior subspaces $\Sp$, $\V$ and $\W$ (\ie not necessarily SI), the
coefficient sequence $d$ can be obtained by applying the transformation
\begin{equation}
H_{\text{sub}} = (V^*V)^{-1}V^*W(S^*W)^{-1}
\end{equation}
on the sample sequence $c$, where $W$ is the synthesis operator associated to the prior functions $\{w_n(t)\}$. This
general formulation will be used in the next sections to derive a subspace-prior recovery technique for
non-invertible Gabor transforms.


\section{Integer under-sampling}\label{sec:GaborInt}

In this section we address the problem of recovering a signal $f(t)$ from its non-invertible Gabor transform
coefficients $\{c_{k,l}\}$, given by (\ref{eq:analysis}), using a pre-specified synthesis window $v(t)$. We focus on
prior-free approaches that do not take into account any knowledge on the signal $f(t)$. Specifically, here we employ
the consistency and minimax-regret methods discussed in the previous section to the Gabor scenario. To emphasize the
commonalities with respect to the SI sampling case, and to retain simplicity, we begin the discussion with the case
of integer under-sampling ($p=1$). In the next section we generalize the results to arbitrary $p$.

\subsection{Consistent synthesis}
\label{sec:consGaborInteger}

In the Gabor transform, the sampling (analysis) space $\Sp$ is spanned by the Gabor system $\G(s,a,b)$ and the
reconstruction (synthesis) space $\V$ is the span of $\G(v,a,b)$. As discussed in Section \ref{sec:consSI},
consistent reconstruction is possible if $\Sp^{\perp} \oplus \A = \Lt(\R)$. In the case of SI spaces, this direct-sum
condition translates to the requirement that the cross-correlation sequence $\{\inner{s(t)}{v(t-n)}\}_{n\in\Z}$ has
an inverse in the convolution algebra $\lo(\Z^2,\ast)$. A similar condition is true in the setting of Gabor spaces.

The next proposition characterizes the class of pairs of analysis and synthesis windows satisfying the direct-sum
requirement in the integer under-sampling regime.
\begin{proposition}\label{prop:direct sum p=1}
Assume that $\G(s,a,b)$ and $\G(v,a,b)$ are Riesz sequences that span the spaces $\Sp$ and $\V$ respectively, and
$ab=q\in\mathbb{N}$. Then $\Sp^{\perp} \oplus \V = \Lt(\R)$ if and only if the function $\Phi^{sv}(\om,x)$, defined
as
\begin{equation}\label{eq:phi_sv}
\Phi^{sv}(\om,x) = \sum_{k,l\in\Z} r_{sv}[k,l] e^{-2\pi i (blx + ak\om)} = (\F r_{sv})(\om,x),
\end{equation}
is nonzero for all $(\om,x)\in [0,1/a)\times [0,1/b)$. Here,
\begin{equation}
r_{sv}[k,l] = \inner{v}{M_{bn} T_{am} s}\label{eq:r_sv}
\end{equation}
is the Gabor transform of the synthesis window $v(t)$.
\end{proposition}
\begin{proof}
It was shown in \cite{EW05}, for general Hilbert spaces, that if $\Sp$ and $\V$ are spanned by Riesz bases
$\G(s,a,b)$ and $\G(v,a,b)$ respectively, then $\Sp^{\perp} \oplus \V = \Lt(\R)$ if and only if the operator
$S^{\ast} V$ is continuously invertible on $\lt$. Here, $S^*$ and $V$ are the analysis and synthesis operators
associated with $\G(s,a,b)$ and $\G(v,a,b)$, respectively. By definition, for any sequence $c\in\lt(\Zt)$
\begin{align}
(S^{\ast} Vc)[k,l] &= \inner{\sum_{m,n\in\Z} c_{m,n}M_{bn} T_{am} v}{M_{bl} T_{ak} s} \nonumber\\
&= \sum_{m,n\in\Z} c_{m,n} \inner{v}{M_{bl-bn} T_{ak-am} s} \nonumber\\
&= \sum_{m,n\in\Z} c_{k-m,l-n} \inner{v}{M_{bn} T_{am} s} \nonumber\\
&= (\conv{r_{sv}}{c})[k,l].
\end{align}
Hence, the operator $S^{\ast} V$ is simply a 2D convolution operator with kernel $r_{sv}[k,l] = \inner{v}{M_{bn}
T_{am} s}$ and $S^{\ast} V$ is invertible if and only if $r_{sv}[k,l]$ is invertible in the convolution algebra
$\lo(\Zt,\ast)$. As shown in Section \ref{sec:GaborRiesz}, this sequence has a representation $\Phi^{sv}(\om,x)$,
defined by (\ref{eq:phiHdef}), which is its 2D DTFT in the case $p=1$. A sequence is invertible with respect to
convolution if and only if its DTFT has no zeros. Therefore, $r_{sv}[k,l]$ is invertible if and only if
$\Phi^{sv}(\om,x)\neq 0$ implying that $\Sp^{\perp} \oplus \V = \Lt(\R)$ if and only if $\Phi^{sv}(\om,x)\neq 0$.
\end{proof}

Assuming that indeed $\Sp^{\perp} \oplus \A = \Lt(\R)$, we know from Section \ref{sec:consSI} that to obtain a
consistent recovery, we must apply the operator $H_{\text{con}}=(S^*V)^{-1}$ on the coefficients $\{c_{k,l}\}$ prior
to synthesis. In the proof of Proposition~\ref{prop:direct sum p=1}, we showed that $S^*V$ is a 2D convolution
operator with the kernel $r_{sv}[k,l]$ of (\ref{eq:r_sv}). Therefore, $(S^*V)^{-1}$ corresponds to filtering the
Gabor coefficients with the filter $h_{\text{con}}$ whose 2D DTFT is given by
\begin{equation}\label{eq:HconGabor}
H_{\text{con}}(\om,x) = \frac{1}{\Phi^{sv}(\om,x)}.
\end{equation}
This filter is well defined by Proposition \ref{prop:direct sum p=1} since we assumed that the spaces generated by
$s(t)$ and $v(t)$ satisfy the direct sum condition.

Observe that during the operations of analysis and pre-processing of the Gabor coefficients $c_{k,l}$, we in fact
compute a dual Riesz basis for the reconstruction space $\V$. In case the synthesis and analysis spaces are the same,
namely $\Sp = \V$, we compute the orthogonal dual basis. However, when the spaces are different we compute a general
(oblique) dual Riesz basis for $\V$.
\begin{proposition}
Let $\G(s,a,b)$ and $\G(v,a,b)$ be Riesz sequences that span the spaces $\Sp$ and $\V$ respectively, where $ab$ is an
integer, and assume that $\Sp^{\perp} \oplus \V = \Lt(\R)$. Then a dual Riesz basis for the space $\V$ is $\G(g,a,b)$
with
\begin{equation}
g(t) = \sum_{m,n \in \Z} \overline{h_{\text{con}}[m,n]}T_{-am}M_{-bn}s(t) \, \in \Sp.\label{eq:dual_g}
\end{equation}
where $h_{\text{con}}[m,n]$ is the inverse of $r_{sv}[k,l]$ with respect to $\ast$.
\end{proposition}
\begin{proof}
Any signal in $\V$ can be recovered from the corrected coefficients
$d_{k,l} = (\conv{h_{\text{con}}}{c})[k,l]$ via $f(t) =
\sum_{k,l\in\Z} d_{k,l} M_{bl} T_{ak} v(t)$, where $c_{k,l}$ is as
in (\ref{eq:analysis}). Therefore, we may view this sequence as the
coefficients in a basis expansion. To obtain the corresponding basis
we note that by combining the effects of the analysis window $s(t)$
and the correction filter $H_{\text{con}}$ of (\ref{eq:HconGabor}),
the expansion coefficients can be equivalently expressed as $d_{k,l}
= \inner{f}{M_{bl}\,T_{ak}\,g}$ where
\begin{equation}
g(t) = \sum_{m,n \in \Z} \overline{h_{\text{con}}[m,n]}T_{-am}M_{-bn}s(t) \, \in \Sp.
\end{equation}
Indeed,
\begin{align}
\inner{f}{M_{bl}T_{ak},g}
&= \inner{f}{M_{bl}T_{ak} \left(\sum_{m,n \in \Z}\overline{h_{\text{con}}[m,n]}T_{-am}M_{-bn}s \right)} \nonumber\\
&= \inner{f}{\sum_{m,n \in \Z} \overline{h_{\text{con}}[m,n]}M_{bl}T_{ak}T_{-am}M_{-bn}s} \nonumber\\
&= \inner{f}{\sum_{m,n \in \Z} \overline{h_{\text{con}}[m,n]} M_{bl-bn}T_{ak-am}s}\nonumber\\
&= \sum_{m,n \in \Z} h_{\text{con}}[m,n]\inner{f}{M_{bl-bn}T_{ak-am}s}\nonumber\\
&= \sum_{m,n \in \Z} h_{\text{con}}[m,n]c_{k-m,l-n}\nonumber\\
&= (\conv{h_{\text{con}}}{c})[k,l] = d_{k,l}.
\end{align}
Therefore, any $f\in \V$ can be written as
\begin{equation}
f(t) = \sum_{k,l\in\Z} \inner{f}{M_{bl}T_{ak}g}M_{bl}T_{ak}v(t).
\end{equation}
It can be easily verified, by Proposition \ref{prop:equivalent basis}, that $\G(g,a,b)$ is an equivalent Riesz basis
for $\Sp$. Furthermore, it can be checked that
\begin{equation}
\inner{M_{bl}T_{ak}v}{M_{bn}T_{am}g} = \de_{m-k}\de_{n-l},
\end{equation}
implying that $\G(g,a,b)$ is a dual Riesz basis to $\G(v,a,b)$.
\end{proof}

\subsection{Minimax regret synthesis}
\label{sec:minimaxGaborInteger}

We now wish to develop a minimax-regret recovery method, similar to the SI sampling case of
Section~\ref{sec:minimaxSI}. Specifically, we would like to produce a recovery $\tilde{f}(t)$ for which the
worst-case regret $\|\tilde{f}-f\|^2 -\|P_{\V^\perp} f\|^2$ over all bounded-norm signals $f(t)$ consistent with the
given Gabor coefficients $\{c_{k,l}\}$, is minimal. As mentioned in Section~\ref{sec:minimaxSI}, the minimax-regret
reconstruction can be obtained by applying the operator $H_{\text{mx}}=(V^*V)^{-1}S^*V(S^*S)^{-1}$ on the Gabor
coefficients $c_{k,l}$ prior to synthesis.

From Section~\ref{sec:consGaborInteger} we know that when $p=1$, the operators $V^*V$, $S^*V$ and $S^*S$ correspond
to 2D convolutions with the kernels $r_{vv}[k,l]$, $r_{sv}[k,l]$ and $r_{ss}[k,l]$ respectively, which are given by
(\ref{eq:r_sv}) with the appropriate substitution of $s(t)$ and $v(t)$. Therefore, the minimax-regret recovery is
obtained by filtering the Gabor coefficients $c_{k,l}$ with the 2D filter $h_{\text{mx}}$, whose DTFT is given by
\begin{equation}\label{eq:HmxGabor}
H_{\text{mx}}(\om,x) = \frac{\Phi^{sv}(\om,x)}{\Phi^{ss}(\om,x)\Phi^{vv}(\om,x)}.
\end{equation}
Here, $\Phi^{sv}(\om,x)$, $\Phi^{ss}(\om,x)$, and $\Phi^{vv}(\om,x)$ are the 2D DTFTs of $r_{sv}[k,l]$, $r_{ss}[k,l]$
and $r_{vv}[k,l]$ respectively. This filter is well defined by Proposition~\ref{prop:RieszPhi} since we assumed that
$s(t)$ and $v(t)$ generate Riesz bases for their span.

\subsection{Efficient implementation}
As we have seen, the two reconstruction approaches discussed above are based on 2D filtering of the Gabor transform
$c_{k,l}$ prior to synthesis. A significant reduction in computation can be achieved in cases where the 2D correction
filter is a separable function of $k$ and $l$, namely when $h_{k,l}=u_k v_l$ for two sequences $u_k$ and $v_k$. In
these situations, one can first apply the 1D filter $u_k$ on each of the rows of $c_{k,l}$ (\ie along the time
direction), and then apply the 1D filter $v_l$ on each of the columns (along the frequency direction). If, for
example, $h_{k,l}$ is a separable finite-impulse-response (FIR) filter with $N\times N$ nonzero coefficients, then
direct application of it requires $N^2$ multiplications per output coefficient, whereas only $2N$ multiplications
suffice when implementing it using two 1D filtering operations.

Separable correction filters emerge when the cross-correlation sequences involved are separable functions of $k$ and
$l$. One such example is the case where $s(t)$ and $v(t)$ are Gaussian windows with variances $\sigma_s^2$ and
$\sigma_v^2$ respectively and $ab\sigma_s^2/(\sigma_s^2+\sigma_v^2)$ is an integer (recall that we also require that
$ab$ be an integer). Then $r_{sv}[k,l]$, $r_{ss}[k,l]$, and $r_{vv}[k,l]$ are all separable functions of $k$ and $l$,
so that both the consistent and the minimax-regret filters are separable. More details on non-invertible
Gaussian-window Gabor transforms are given in Section~\ref{sec:example}.


\section{Rational under-sampling}\label{sec:GaborRational}

We now generalize the results of the previous section to the case where the product $ab$ is not an integer, but
rather some rational number $q/p$ with $p$ and $q$ relatively prime. The main difficulty here is the fact that the
time-frequency shift operators do not commute when $p\neq 1$. Therefore, instead of standard convolution we will be
faced with a twisted convolution, which is a noncommutative operation. This makes the techniques from Fourier theory
inapplicable in a straightforward manner.

\subsection{Consistent synthesis}
\label{sec:consGaborRational}

Obtaining a reconstruction $\tilde{f}(t)$, which is consistent with the Gabor representation $c_{k,l}$ of $f(t)$, is
possible if $\Sp^\perp \oplus \V = \Lt(\R)$. As we have seen in Proposition \ref{prop:direct sum p=1}, in the integer
under-sampling case $p=1$ the direct sum condition translates to the requirement that the cross-correlation sequence
$r_{sv}[k,l]$ be invertible in the convolution algebra $\lo(\Z^2,\ast)$. In the setting of rational under-sampling,
we have the following.

\begin{proposition}\label{prop:direct sum}
Assume that $\G(s,a,b)$ and $\G(v,a,b)$ are Riesz sequences that span the spaces $\Sp$ and $\V$ respectively, and
$ab=q/p$ with $p$ and $q$ relatively prime. Then $\Sp^{\perp} \oplus \V = \Lt(\R)$ if and only if the $(p\times
p)$-matrix-valued function $\bPhi^{sv}(\om,x)$ with entries defined as
\begin{equation}
\bPhi^{sv}_{m,n}(\om,x) = \sum_{k,l\in\Z} r_{sv}[n-m+pk,l] e^{-2\pi iablm}e^{-2\pi i (blx+ak\om)}\quad
m,n=0,\ldots,p-1.
\end{equation}
is invertible for all $(\om,x)\in [0,1/a)\times [0,1/b)$, which is equivalent to $\det(\bPhi(\om,x))\neq 0$ for all
$(\om,x)$.
\end{proposition}
\begin{proof}
The proof is similar to the proof of Proposition \ref{prop:direct sum p=1}. Since $s(t)$ and $v(t)$ generate Riesz
bases for $\Sp$ and $\V$ respectively, the condition $\Sp^{\perp} \oplus \V = \Lt(\R)$ is satisfied if and only if
the operator $S^{\ast} V$ is continuously invertible on $\lt$, where $S^*$ and $V$ are the analysis and synthesis
operators associated to $\G(s,a,b)$ and $\G(v,a,b)$ respectively. By definition, for any sequence $c\in\lt(\Zt)$, we
have
\begin{align*}
(S^{\ast} Vc)[k,l] &= \inner{\sum_{m,n\in\Z} c_{m,n}M_{bn} T_{am} v}{M_{bl} T_{ak} s}\\
&= \sum_{m,n\in\Z} c_{m,n} \inner{v}{M_{bl-bn} T_{ak-am} s}e^{-2\pi i (bl-bn)am} \\
&= \sum_{m,n\in\Z} c_{k-m,l-n} \inner{v}{M_{bn} T_{am} s} e^{-2\pi i ab(k-m)n}\\
&= (\twc{r_{sv}}{c})[k,l].
\end{align*}
Therefore, $S^{\ast} V$ is a twisted convolution operator with kernel
\begin{equation}
r_{sv}[k,l] = \inner{v}{M_{bl} T_{ak} s},
\end{equation}
and $S^{\ast} V$ is invertible if and only if $r_{sv}[k,l]$ is invertible in the twisted convolution algebra
$\lo(\Zt,\nat)$. As shown in Section~\ref{sec:GaborRiesz}, this sequence has a representation $\bPhi^{sv}(\om,x)$
defined by (\ref{eq:phiHdef}) and so is invertible if and only if this matrix is invertible. Therefore, $\Sp^{\perp}
\oplus \V = \Lt(\R)$ if and only if $\bPhi^{sv}(\om,x)$ is invertible for all $\om$ and $x$.
\end{proof}

Note that for $p=1$, the above proposition reduces to Proposition \ref{prop:direct sum p=1}. When $p\neq1$, we
conclude from Proposition~\ref{prop:direct sum} that the direct sum condition translates to the requirement that
$r_{sv}[k,l]$ be invertible in the twisted convolution algebra, which can be checked by analyzing its
$\bPhi$-representation. An alternative method for checking weather $r_{sv}[k,l]$ is invertible with respect to
$\nat$, is presented in Section \ref{sec:twc}. It involves only the sequence $r_{sv}[k,l]$ without introducing the
continuous variables $\om$ and $x$, making it more attractive in some cases.

As in Section~\ref{sec:consGaborInteger}, to obtain a consistent recovery $\tilde{f}(t)$, we have to apply the
operator $H_{\text{con}}=(S^*V)^{-1}$ to the Gabor coefficients $c_{k,l}$. However, as opposed to the case $p=1$,
where $H_{\text{con}}$ was a standard convolution operator, here it corresponds to a twisted convolution operation.
This is due to the fact that time-frequency shift operators do not commute for $p\neq 1$. Specifically, in the proof
of Proposition \ref{prop:direct sum}, it was shown that $S^{\ast} V$ corresponds to twisted convolution with
$r_{sv}[k,l]$. Therefore, $(S^*V)^{-1}$ corresponds to twisted convolution with the sequence $r_{sv}^{-1}[k,l]$,
which is the inverse of $r_{sv}[k,l]$ in the twisted convolution algebra $\lo(\Z^2,\nat)$. This inverse exists, since
we assumed that the spaces generated by $s(t)$ and $v(t)$ satisfy the direct-sum condition, and it will be shown in
the next section how to construct it.

One can write the twisted convolution relation between the Gabor transform $c_{k,l}$ and the expansion coefficients
$d_{k,l}$ in terms of their $\bPhi$-representations. Specifically, since $d=(S^*V)^{-1}c$, we have
$c_{k,l}=(r_{sv}\nat d)[k,l]$ and therefore
\begin{equation}
\bPhi^c(\om,x) = \bPhi^d(\om,x) \bPhi^{sv}(\om,x),
\end{equation}
where $\bPhi^c(\om,x)$, $\bPhi^d(\om,x)$ and $\bPhi^{sv}(\om,x)$ are the $p\times p$-matrix-valued
$\bPhi$-representations of the sequences $c_{k,l}$, $d_{k,l}$ and $r_{sv}[k,l]$ respectively, defined in
(\ref{eq:Phi-def}). Therefore, to obtain the sequence $d_{k,l}$ from the Gabor coefficients $c_{k,l}$, we apply a
twisted convolution filter, whose $\bPhi$ function is
\begin{equation}\label{eq:H_sv}
\boldsymbol{H}_{\text{con}}(\om,x) = \bPhi^{sv}(\om,x)^{-1}.
\end{equation}
The twisted convolution operation can be modeled as a filter bank which is specified by the convolutional inverse of
$r_{sv}[k,l]$, as we show in Section~\ref{sec:twc}.

During the operations of sampling and pre-processing of the samples $c_{k,l}$ we in fact compute a dual Riesz basis
for the synthesis space $\V$. If the synthesis and analysis spaces are the same, namely $\Sp = \V$, we compute the
orthogonal dual basis. However, when the spaces are different we compute a general (oblique) dual Riesz basis for
$\V$.
\begin{proposition}
Let $\G(s,a,b)$ and $\G(v,a,b)$ be Riesz sequences that span the spaces $\Sp$ and $\V$ respectively, and $ab=q/p$
with $p$ and $q$ relatively prime. Assume that $\Sp^{\perp} \oplus \V = \Lt(\R)$. Then a dual Riesz basis for the
space $\V$ is $\G(g,a,b)$ with
\begin{equation}
g(t) = \sum_{m,n \in \Z} \overline{h_{\text{con}}[m,n]}T_{-am}M_{-bn}s(t) \in \Sp\,,
\end{equation}
where $h_{\text{con}}[m,n]$ is the inverse of $r_{sv}[k,l]$ with respect to $\nat$.
\end{proposition}
\begin{proof}
Any signal in $\V$, that has been sampled with the Riesz sequence $\G(s,a,b)$ resulting in the coefficients
$c_{k,l}$ given by (\ref{eq:analysis}), can be recovered from the corrected samples $d_{k,l} =
(\twc{h_{\text{con}}}{c})[k,l]$, where $h_{\text{con}}[k,l]=r^{-1}_{sv}[k,l]$ is the inverse of $r_{sv}[k,l]$ with
respect to $\nat$, via $f(t) = \sum_{m,n\in\Z}d_{k,l}M_{bl}T_{ak}v(t)$. This sequence may be viewed as the
coefficients in a basis expansion. To obtain the corresponding basis we note that by combining the effects of the
analysis window $s(t)$ and the correction twisted-convolution filter $h_{\text{con}}[k,l]$, the expansion
coefficients can be equivalently expressed as $d_{k,l} = \inner{f}{M_{bl}T_{ak}g}$ where
\begin{equation}
g(t) = \sum_{m,n \in \Z} \overline{h_{\text{con}}[m,n]}T_{-am}M_{-bn}s(t) \in \Sp.
\end{equation}
Indeed,
\begin{align}
\inner{f}{M_{bl}T_{ak}g} &= \inner{f}{M_{bl}T_{ak} \left(\sum_{m,n \in \Z} \overline{h_{\text{con}}[m,n]}T_{-am}M_{-bn}s \right)} \nonumber \\
&= \inner{f}{\sum_{m,n \in \Z} \overline{h_{\text{con}}[m,n]}M_{bl}T_{ak}T_{-am}M_{-bn}s} \nonumber\\
&= \inner{f}{\sum_{m,n \in \Z} \overline{h_{\text{con}}[m,n]}e^{2\pi iab(k-m)n}M_{bl-bn}T_{ak-am}s},
\end{align}
and using the linearity of the inner product, we have
\begin{align}
\inner{f}{M_{bl}T_{ak}g} &= \sum_{m,n \in \Z} h_{\text{con}}[m,n]e^{-2\pi i ab(k-m)n}\inner{f}{M_{bl-bn}T_{ak-am}s} \nonumber\\
&= \sum_{m,n \in \Z} h_{\text{con}}[m,n]e^{-2\pi i ab(k-m)n}c_{k-m,l-n} \nonumber\\
&= (\twc{h_{\text{con}}}{c})[k,l] = d_{k,l}.
\end{align}
Therefore, any $f\in \V$ can be written as
\begin{equation}
f(t) = \sum_{k,l\in\Z} \inner{f}{M_{bl}T_{ak}g}M_{bl}T_{ak}v(t).
\end{equation}

It can be easily verified, by Proposition \ref{prop:equivalent basis}, that $\G(g,a,b)$ is an equivalent Riesz basis
for $\Sp$. Now, for it to be a dual Riesz basis to $\G(v,a,b)$ we need to check that
\begin{equation}
\inner{M_{bl}T_{ak}v}{M_{bn}T_{am}g} = \de_{m-k}\de_{n-l}.
\end{equation}
Indeed,
\begin{align*}
\inner{M_{bl}T_{ak}v}{M_{bn}T_{am}g} &= e^{2\pi i ab(l-n)k}\inner{v}{M_{b(n-l)}T_{a(m-k)}g}\\
&= e^{2\pi i ab(l-n)k}\sum_{x,y\in\Z} h_{\text{con}}[x,y] \inner{v}{M_{b(n-l-y)}T_{a(m-k-x)}s} e^{-2\pi i ab(m-k-x)y}\\
&= e^{2\pi i ab(l-n)k}\sum_{x,y\in\Z} h_{\text{con}}[x,y]r_{sv}[m-k-x,n-l-y]e^{-2\pi i ab(m-k-x)y}\\
&= e^{2\pi i ab(l-n)k}(\twc{h_{\text{con}}}{r_{sv}})[m-k,n-l] = \de_{m-k}\de_{n-l},
\end{align*}
where we used the fact that $M_{bn}T_{am}g(t) = \sum_{x,y\in\Z} \overline{h_{\text{con}}[x,y]}e^{2\pi i
ab(m-x)y}M_{b(n-y)}T_{a(m-x)}s(t)$ and that $h_{\text{con}}$ is the inverse of $r_{sv}$ with respect to $\nat$.
\end{proof}

\subsection{Minimax regret synthesis}
\label{sec:minimaxGaborRational}

Next, we develop a minimax-regret reconstruction method for non-invertible Gabor transforms with rational
under-sampling. Our goal here, as in Section~\ref{sec:minimaxGaborInteger}, is to minimize the worst case regret
$\max_{f\in\B} \{\|\tilde{f}-f\|^2 -\|P_{\V^\perp} f\|^2\}$, where $\B$ is the set of bounded-norm signals whose
Gabor coefficients coincide with $c_{k,l}$. As discussed in Section~\ref{sec:minimaxSI}, The recovery $\tilde{f}$
attaining the minimum can be obtained by applying the operator $H_{\text{mx}}=(V^*V)^{-1}S^*V(S^*S)^{-1}$ on the
Gabor coefficients $c_{k,l}$ prior to synthesis. However, as opposed to the integer under-sampling case discussed in
Section~\ref{sec:minimaxGaborInteger}, where $V^*V$, $S^*V$, and $S^*S$ were convolution operators, here they
correspond to twisted convolutions with $r_{vv}[k,l]$, $r_{sv}[k,l]$ and $r_{ss}[k,l]$ respectively. Therefore, to
obtain the sequence $d_{k,l}$, we apply a twisted convolution filter on the Gabor coefficients $c_{k,l}$, whose
impulse response is
\begin{equation}
h_{\text{mx}}[k,l] = \left(\twc{r_{vv}^{-1}}{\twc{r_{sv}}{r_{ss}^{-1}}}\right)[k,l].
\end{equation}
Here, $r_{vv}^{-1}[k,l]$ and $r_{ss}^{-1}[k,l]$ are the inverses of $r_{vv}[k,l]$ and $r_{ss}[k,l]$ with respect to
$\nat$. Consequently, the $\bPhi$ function of the minimax-regret filter is given by
\begin{equation}\label{eq:H_mx}
\boldsymbol{H}_{\text{mx}}(\om,x) = \bPhi^{ss}(\om,x)^{-1}\bPhi^{sv}(\om,x)\bPhi^{vv}(\om,x)^{-1},
\end{equation}
where $\bPhi^{ss}(\om,x)$, $\bPhi^{sv}(\om,x)$, and $\bPhi^{vv}(\om,x)$ are the $\bPhi$-representations of
$r_{ss}[k,l]$, $r_{sv}[k,l]$ and $r_{vv}[k,l]$ respectively.

\subsection{Extension to symplectic lattices}
\label{sec:GaborSymplectic}

Throughout the current and previous sections, we considered a special type of sampling points in the time-frequency
plane, called separable lattices $\La = a\Z \times b\Z$. However, with the help of metaplectic operators, these
results carry over to the more general class of lattices, called symplectic lattices. A lattice $\La_s \subseteq
\R^2$ is called symplectic, if one can write $\La_s = \mathcal{D}\La$ where $\La$ is a separable lattice and
$\mathcal{D}\in GL_2(\R)$, meaning it is an invertible $2\times 2$ matrix with determinant $1$ \cite{G01}. To every
$\mathcal{D} \in GL_2(\R)$ there corresponds a unitary operator $\mu(\mathcal{D})$, called metaplectic, acting on
$\Lt(\R)$. One can show that a Gabor system on a symplectic lattice is unitarily equivalent to a Gabor system on a
separable lattice under $\mu(\mathcal{D})$, that is $\G(g,\La_s)$ is a frame/Riesz basis if and only if
$\G(\mu(\mathcal{D})^{-1}g,\La)$ is a frame/Riesz basis, and
\begin{equation}
\G(g,\La_s) = \mu(\mathcal{D})\G(\mu(\mathcal{D})^{-1}g,\La)\,.
\end{equation}
Therefore, instead of considering a representation of $f(t)$ in $\overline{\rm span}\{ g_{\la} \}_{\la\in\La_s}$ one
can look at the representation of $f(t)$ in $\overline{\rm span}\{ \mu(\mathcal{D})^{-1}g_{\la} \}_{\la\in\La}$. For
more details see \cite{G01}.


\section{Subspace-prior synthesis}\label{sec:subspaceGabor}
\label{sec:subspaceGabor}

In the previous two sections we attempted to recover a signal from its non-invertible Gabor representations without
using any prior knowledge on the signal. When such knowledge is available, it can significantly reduce the
reconstruction error and in some cases even lead to perfect recovery. A common prior in sampling theory is that the
signal to be recovered lies in some SI subspace of $\Lt$, namely that it can be written as
\begin{equation}
f(t) = \sum_{k\in\Z} d_k T_{ak}w(t) = \sum_{k\in\Z} d_k w(t-ak)
\end{equation}
with some norm-bounded sequence $\{d_k\}$ and some window $w(t)$. This model can quite accurately describe many types of natural signals, which exhibit a
certain degree of smoothness. For example, the class of bandlimited signals is the SI space generated by the sinc
window. The class of splines of degree $N$ also follows this description with $w(t)$ being the B-spline function of
degree $N$.

Here, we would like to generalize the SI-prior setting to Gabor spaces, which we also term in this context
\emph{shift-and-modulation-invariant} (SMI) spaces. We will use these spaces as priors on our input signals, in order
to recover them from their non-invertible Gabor transform. An SMI subspace $\W\subseteq\Lt$ is the set of signals
that can be represented in the form
\begin{equation}\label{eq:finGaborW}
f(t) = \sum_{k,l\in\Z} h_{k,l}M_{bl}T_{ak}w(t),
\end{equation}
for some sequence $h_{k,l}$ in $\lt(\Zt)$, where $w(t)$ is an arbitrary window in $\SO$. In other words, $\W$ is the
closed linear span of the Gabor system $\G(w,a,b)$. Our choice of terminology follows from the fact that if $f(t)$
lies in $\W$, then the function $M_{bl}T_{ak}f(t)$ is also an element of $\W$ for every fixed $k,l\in\Z$. Indeed, let
$f(t)=\sum_{m,n\in\Z} h_{m,n}M_{bn}T_{am}w(t)$ for some sequence $h_{m,n}$, then
\begin{align}
M_{bl}T_{ak}f(t) &= M_{bl}T_{ak}\left( \sum_{m,n\in\Z} h_{m,n}M_{bn}T_{am}w(t) \right)\nonumber\\
&= \sum_{m,n\in\Z} h_{m,n}M_{bl}T_{ak}M_{bn}T_{am}w(t)\nonumber\\
&= \sum_{m,n\in\Z} h_{m,n} e^{-2\pi i abkn} M_{b(n+l)}T_{a(m+k)}w(t)\nonumber\\
&= \sum_{m,n\in\Z} h_{m-k,n-l}e^{-2\pi i ab(n-l)k}M_{bn}T_{am}w(t)\nonumber\\
&= \sum_{m,n\in\Z} d_{m,n}M_{bn}T_{am}w(t) \in \W,
\end{align}
where $d_{m,n} = h_{m-k,n-l}e^{-2\pi i ab(n-l)k}$. The same holds for $T_{ak}M_{bl}f(t)$.

Our setting is thus as follows. We assume that $f(t)$ lies in some SMI space $\W$, generated by $\G(w,a,b)$, which we
term the prior space, and that we are given the Gabor coefficients $c_{k,l}$ of $f(t)$, which were computed with the
analysis window $s(t)$. Our goal is to produce a recovery $\tilde{f}(t)$ using the synthesis window $v(t)$. Clearly,
if $\W$ does not coincide with our synthesis space $\V$, then the reconstruction $\tilde{f}(t)$ cannot equal $f(t)$.
The interesting question is whether we can obtain the best possible recovery, which is the orthogonal projection
$\tilde{f}=P_{\V}f$, from the Gabor coefficients $c_{k,l}$ of $f(t)$. As above, we discuss the integer and rational
under-sampling cases separately.

\subsection{Integer under-sampling}
\label{sec:subspaceGaborInteger}

As discussed in Section~\ref{sec:subspaceSI}, if the analysis and prior spaces satisfy $\Sp^{\perp} \oplus \W =
\Lt(\R)$, then the recovery $\tilde{f}=P_{\V}f$ can be generated by applying the operator $H_{\text{sub}} =
(V^*V)^{-1}V^*W(S^*W)^{-1}$ on the Gabor coefficients $c_{k,l}$ prior to synthesis. From Proposition~\ref{prop:direct
sum p=1} we know that this direct-sum condition is satisfied if and only if $\Phi^{sw}(\om,x)\neq 0$ for all $\om$
and $x$, where $\Phi^{sw}(\om,x)$ is as in (\ref{eq:phi_sv}) with $v(t)$ replaced by $w(t)$. The operators $V^*V$,
$V^*W$ and $S^*W$ correspond to 2D convolutions with the kernels $r_{vv}[k,l]$, $r_{vw}[k,l]$ and $r_{ww}[k,l]$
respectively, which are given by (\ref{eq:r_sv}) with the appropriate substitution of $s(t)$, $v(t)$ and $w(t)$.
Hence, the operator $H_{\text{sub}}$ corresponds to 2D convolution with the filter $h_{\text{sub}}$, whose 2D DTFT is
given by
\begin{equation}\label{eq:HsubGabor}
H_{\text{sub}}(\om,x) = \frac{\Phi^{vw}(\om,x)}{\Phi^{sw}(\om,x)\Phi^{vv}(\om,x)},
\end{equation}
where $\Phi^{vw}(\om,x)$, $\Phi^{sw}(\om,x)$, and $\Phi^{vv}(\om,x)$ are the 2D DTFTs of $r_{vw}[k,l]$, $r_{sw}[k,l]$
and $r_{vv}[k,l]$ respectively.

When the synthesis space $\V$ coincides with the prior space $\W$, we have $H_{\text{sub}} =
(V^*V)^{-1}V^*W(S^*W)^{-1}=(S^*W)^{-1}$, so that the correction filter is the same as in the consistency approach of
section~\ref{sec:consGaborInteger}. In this case, the direct-sum condition (namely the invertibility of the operator
$S^*W$) guarantees perfect recovery of $f(t)$. To see this, note that any $f\in\W$ can be written as $f=Wd$ for some
sequence $d_{k,l}$, so that the Gabor coefficients $c_{k,l}$ are given by $c=S^*f=S^*Wd$. Therefore, the expansion
coefficients can be perfectly recovered using $d=(S^*W)^{-1}c$. This property is, of course, independent of the
sampling lattice and holds true also in the rational under-sampling regime.

\subsection{Rational under-sampling}
\label{sec:subspaceGaborRational}

We now extend the subspace-prior approach to the rational under-sampling regime. As before, we assume that the input
$f(t)$ can be expressed in the form (\ref{eq:finGaborW}) for some sequence $h_{k,l}$, where $w(t)$ is a given window
in $\SO$. As we have seen, the best possible recovery, which is the orthogonal projection $\tilde{f}=P_{\V}f$, can be
obtained if the analysis and prior spaces satisfy $\Sp^\perp \oplus \W = \Lt(\R)$, which in our case is equivalent to
$r_{sw}[k,l]$ being invertible with respect to twisted convolution. In this case, $\tilde{f}=P_\V f$ can be produced
by applying the operator $H_{\text{sub}}=(V^*V)^{-1}V^*W(S^*W)^{-1}$ on the Gabor transform $c_{k,l}$ prior to
reconstruction. The operators $V^*V$, $V^*W$, and $S^*W$ correspond to twisted convolution with the kernels
$r_{vv}[k,l]$, $r_{v,w}[k,l]$ and $r_{s,w}[k,l]$ respectively. Therefore, $H_{\text{sub}}$ corresponds to twisted
convolution with
\begin{equation}
h_{\text{sub}}[k,l] = \left(\twc{r_{vv}^{-1}}{\twc{r_{vw}}{r_{sw}^{-1}}}\right)[k,l],
\end{equation}
where, $r_{vv}^{-1}[k,l]$ and $r_{sw}^{-1}[k,l]$ are the inverses of $r_{vv}[k,l]$ and $r_{sw}[k,l]$ with respect to
$\nat$. Consequently, the $\bPhi$ function of the subspace-prior filter is given by
\begin{equation}\label{eq:H_sub}
\boldsymbol{H}_{\text{mx}}(\om,x) = \bPhi^{sw}(\om,x)^{-1}\bPhi^{vw}(\om,x)\bPhi^{vv}(\om,x)^{-1},
\end{equation}
where $\bPhi^{sw}(\om,x)$, $\bPhi^{vw}(\om,x)$, and $\bPhi^{vv}(\om,x)$ are the $\bPhi$-representations of
$r_{sw}[k,l]$, $r_{vw}[k,l]$ and $r_{vv}[k,l]$ respectively.


\section{Twisted Convolution}\label{sec:twc}

In the previous sections, we saw that in order to process the samples $c_{m,n}$ one needs the inverse of certain
cross-correlation sequences with respect to $\nat$. In this section we show how to obtain explicitly the inverse of
some sequence $d_{k,l}$ with respect to twisted convolution with parameter $ab$. This depends very much on $ab$. If
$ab \in \mathbb{N}$, then the twisted convolution is a standard convolution, and the Fourier transform can be used to
compute the inverse of $d_{k,l}$. If $ab = q/p$, then one can use the construction derived in \cite{EMW07}, which
breaks the problem into computing inverses of several sequences with respect to standard convolution. We now briefly
review this method. For the proofs and more detailed explanations, we refer the reader to the original paper.

Let $d_{k,l}$ be a sequence in $\lo(\Zt)$. We create $p^2$ new sequences out of $d_{k,l}$, defined as
\begin{equation}\label{eq:seq_decomp}
(d^{r,s})_{k,l} = d_{k,l}\sum_{m\in\Z}\sum_{n\in\Z}\delta[k-r-pm,l-s-pn],
\end{equation}
where $r,s=0,1,\ldots,p-1$. 
It is easy to see that the sequence $d^{r,s}$ is supported on the coset $(r + p\Z) \times (s + p\Z)$ and therefore $d
= \sum_{r=0}^{p-1}\sum_{s=0}^{p-1} d^{r,s}$. In the case when $p=2$, out of a sequence $d_{k,l}$ we obtain four
subsequences: $d^{0,0}$ which is supported on $2\Z\times 2\Z$, $d^{0,1}$ supported on $2\Z \times (2\Z +1)$,
$d^{1,0}$ supported on $(2\Z + 1)\times 2\Z$ and $d^{1,1}$ supported on $(2\Z + 1) \times (2\Z + 1)$.

Next, we associate with the sequence $d_{k,l}$ a $p\times p$ matrix $D$ whose entries are sequences in $\lo$:
\begin{equation}\label{eq:seq-matrix}
D_{r,s} = \sum_{m=0}^{p-1} d^{m,r-s} e^{-2\pi i ms q/p},
\end{equation}
where $r-s$ should be interpreted as modulo $p$. This matrix is an element of an algebra $\M$ of $p\times p$ matrices
with multiplication of two matrices $D$ and $E$ given by
\begin{equation}
(D\circledast E)_{r,s} = \sum_{m=0}^{p-1} \conv{D_{r,l}}{E_{l,s}}\,,
\end{equation}
where $\ast$ is a standard convolution. It was shown in \cite{EMW07} that an algebra of such matrices is closed under
taking inverses, meaning that if $D$ is invertible in $\M$ then its inverse is also an element of $\M$ and its
entries are also coming from some sequence in $\lo(\Zt)$. For example, when $p=2$ the above matrix takes the form
\begin{equation*}
D = \left ( \begin{array}{cc}
d^{0,0} + d^{1,0} & d^{0,1} - d^{1,1} \\
d^{0,1} + d^{1,1} & d^{0,0} - d^{1,0}
        \end{array} \right )\,
\end{equation*}
where we used the fact that since $p=2$, $q$ must be odd, and thus $e^{2\pi i msq/2}$ for $m,s=0,1$ takes the values
$1$ and $-1$. Note that summing up the elements of the first column gives us back the sequence $d$.

It was shown in \cite{EMW07} that the invertibility of the sequence $d_{k,l}$ with respect to $\nat$ is equivalent to
the invertibility of the matrix $D$ in this new matrix algebra, which in turn is equivalent to the invertibility of
$\det(D)$ in $\lo(\Zt,\ast)$. Therefore, if $D$ is invertible, its inverse can be computed using Cramer's Rule. That
is the $(r,s)$ entry of $D^{-1}$ is given by
\begin{equation}\label{eq:Dinv}
(D^{-1})_{r,s} = \conv{(\det(D))^{-1}}{\det(D(s,r))},
\end{equation}
where $D(s,r)$ is a $p \times p$ matrix obtained from $D$ by substituting the $s$th row of $D$ with a vector of zeros
having $\de$ on the $r$th position, and the $r$th column with a column of zeros having $\de$ on the $s$th position.
Note that $\det(D)$ is a sequence and its inverse in (\ref{eq:Dinv}) is taken with respect to standard convolution.
For example, when $p=2$ we get
\begin{align}
&D(0,0) = \left (
\begin{array}{cc} \delta & 0 \\ 0 & d^{0,0} - d^{1,0} \end{array}
\right ),\nonumber\\
&D(1,0) = \left (
\begin{array}{cc} 0 & \delta \\ d^{0,1}+d^{1,1} & 0 \end{array} \right ),\nonumber\\
&D(0,1) = \left ( \begin{array}{cc} 0 & d^{0,1} - d^{1,1} \\ \delta & 0 \end{array} \right ),\nonumber\\
&D(1,1) = \left(
\begin{array}{cc} d^{0,0}+d^{1,0} & 0 \\ 0 & \delta \end{array} \right ).
\end{align}
Thus,
\begin{equation*}
D^{-1} = \conv{(\det D)^{-1}}{\left ( \begin{array}{cc}
d^{0,0} - d^{1,0} & -d^{0,1} + d^{1,1} \\
-d^{0,1} - d^{1,1} & d^{0,0} + d^{1,0}
                                      \end{array} \right ) },
\end{equation*}
where $\det(D) = \conv{(d^{0,0}+d^{1,0})}{(d^{0,0}-d^{1,0})} - \conv{(d^{0,1}+d^{1,1})}{(d^{0,1}-d^{1,1})}$. Since
the matrix algebra $\M$ is closed under taking inverses, summing up the elements of the first column of $D^{-1}$
results in some sequence $e_{k,l}$ which is the inverse of $d_{k,l}$ with respect to twisted convolution. Therefore,
it is enough to compute only this column and sum up its entries to get $d^{-1}$. In our example with $p=2$, the
twisted-convolutional-inverse of $d$ equals $\conv{(\det(D))^{-1}}{(d^{0,0} - d^{1,0} - d^{0,1} - d^{1,1})}$.

We mentioned in the previous sections that it is possible to realize twisted convolution with a rational parameter
$ab$ using a filter bank. Indeed, using the decomposition (\ref{eq:seq_decomp}) of the sequences, the twisted
convolution of two sequences $c$ and $d$,
\begin{equation}
(\twc{d}{c})_{m,n} = \sum_{k,l\in\Z} d_{k,l}c_{m-k,n-l} e^{-2\pi i ab(m-k)l} = \sum_{k,l\in\Z} c_{k,l} d_{m-k,n-l}
e^{-2\pi i ab(n-l)k}
\end{equation}
can be written as
\begin{equation}\label{twisted}
(\twc{d}{c}) = \sum_{r,s = 0}^{p-1} \sum_{u,v = 0}^{p-1} (\conv{c^{r,s}}{d^{u-r,v-s}}) e^{-2\pi i (v-s)r q/p} \quad
\mbox{for } u,v=0,1,\ldots,p-1.
\end{equation}
Therefore, as shown in Fig.~\ref{fig:twistFB}, each of the $p^2$ sequences $c^{r,s}$, $r,s=0,1,\ldots,p-1$, is split
into $p^2$ filters associated with the sequences $d^{u,v}$, $u,v=0,1,\ldots,p-1$. Then, $\twc{d}{c}$ is obtained by
summing over the resulting $p^4$ output sequences. Figure~\ref{fig:twistFB} depicts one of the $p^4$ branches, which
corresponds to the indices $r$, $s$ ,$u$ and $v$.

\begin{figure*}\centering
\includegraphics[scale=0.75, trim=0cm 0cm 0cm 0cm]{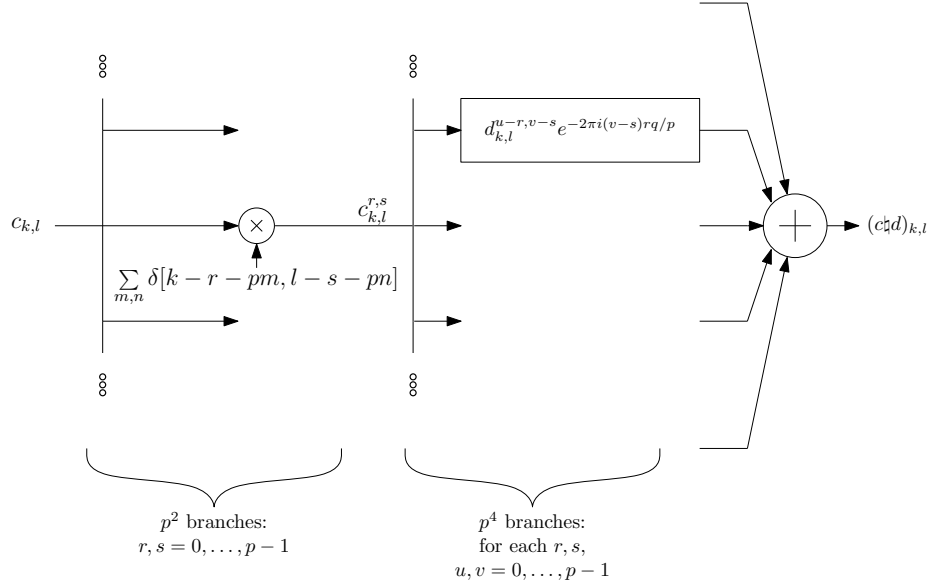}
\caption{A filter-bank realization of twisted convolution between $c_{k,l}$ and $d_{k,l}$.} \label{fig:twistFB}
\end{figure*}

\section{Example: Integer under-sampling with Gaussian windows}
\label{sec:example}

We now demonstrate the prior-free recovery techniques derived in this paper. To retain simplicity we will focus on
the integer under-sampling scenario. In this regime, the smallest amount of information loss occurs when $ab=2$.
Therefore, in our simulations we used $a=1$ and $b=2$. In this setting there are at most half the number of time-frequency coefficients for any given frequency range per time unit, than in any invertible Gabor representation. Consequently, algorithms operating in the Gabor domain (\eg for system identification, speech enhancement, blind source separation, etc.) will benefit from a reduction of at least a factor of $2$ in computational load. On the other hand, we expect the norm of the reconstruction error to be roughly on the order of the signal's norm in the worst-case scenario, since half of the information is lost in such a representation.

For tractability, we will work out the case in which the
analysis and synthesis are both performed with a Gaussian window:
\begin{align}
&s(t)=\frac{1}{\sqrt{2\pi\sigma_s^2}}\exp\left\{-\frac{t^2}{2\sigma_s^2}\right\}\\
&v(t)=\frac{1}{\sqrt{2\pi\sigma_v^2}}\exp\left\{-\frac{t^2}{2\sigma_v^2}\right\}.
\end{align}
In this scenario, the cross-correlation sequence $r_{sv}[k,l] = \inner{v}{M_{bn} T_{am} s}$, has an analytic
expression:
\begin{equation}\label{eq:r_svGauss}
r_{sv}[k,l]=\frac{1}{\sqrt{2\pi\left(\sigma_s^2+\sigma_v^2\right)}}\exp\left\{-\frac{(ak)^2+4\pi^2\sigma_s^2\sigma_v^2(bl)^2}
{2\left(\sigma_s^2+\sigma_v^2\right)}\right\} \exp\left\{\frac{2\pi i \sigma_s^2abkl}{\sigma_s^2+\sigma_v^2}\right\}.
\end{equation}
Similarly, $r_{ss}[k,l]$ and $r_{vv}[k,l]$ can be obtained by replacing $\sigma_s$ by $\sigma_v$ and vice versa.

The $2D$ filter $h_{\text{con}}$ of (\ref{eq:HconGabor}), corresponding to the consistency requirement, is the
convolutional inverse of $r_{ss}[k,l]$. This sequence can be approximated numerically using the discrete Fourier
transform (DFT) of the finite-length sequence $r_{sv}[k,l]$, $(k,l)\in[-K,K]\times[-L,L]$, for some (large) $K$ and
$L$. To compute the filter $h_{\text{mx}}$ of (\ref{eq:HmxGabor}), corresponding to the minimax-regret approach, we
need to invert $r_{ss}[k,l]$ and $r_{vv}[k,l]$, which can be done in a similar manner. Note that both
$h_{\text{con}}$ and $h_{\text{mx}}$ are generally complex sequences. Figure~\ref{fig:hmx_hcon} depicts the modulus
$|h_{\text{con}}|$ and $|h_{\text{mx}}|$ for the case $\sigma_s=0.1$, $\sigma_v=2$, and $ab=2$.


\begin{figure*}\centering
\includegraphics[width=\textwidth, trim=0cm 3cm 0cm 0cm ]{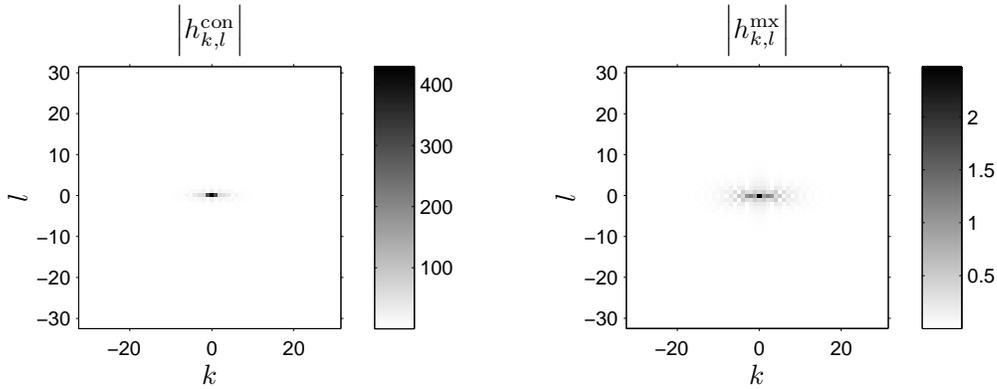}
\caption{The 2D correction filters corresponding to the minimax-regret and consistency methods.} \label{fig:hmx_hcon}
\end{figure*}

To see the effect of these two filters, we now examine the recovery of a chirp signal from its non-invertible Gabor
representation using both methods. Specifically, let
\begin{equation}
f(t)=2\cos\left(t^2\right).
\end{equation}
The Gaussian-window Gabor transform of $f(t)$ has a closed form expression, given by
\begin{align}
c_{k,l}=&\frac{1}{\sqrt{-2i\sigma_s^2+1}}\exp\left\{-\frac{ak(ak+2bl\pi)+2ib^2l^2\pi^2\sigma_s^2}{i+2\sigma_s^2}\right\}\nonumber\\
       &+ \frac{1}{\sqrt{2i\sigma_s^2+1}}\exp\left\{\frac{-ak(ak-2bl\pi)+2ib^2l^2\pi^2\sigma_s^2}{-i+2\sigma_s^2}\right\}.
\end{align}
The signal $f(t)$ and the modulus of its Gabor transform, $|c_{k,l}|$, are shown in Fig.~\ref{fig:f_and_c}. Although
$c_{k,l}$ seems to constitute a good time-frequency representation of $f(t)$, it is certainly not suited to play the
role of the synthesis expansion coefficients $d_{k,l}$. This can be seen in Fig.~\ref{fig:ChirpRec}(a), where
$c_{k,l}$ have been used without modification as expansion coefficients to produce a recovery $\tilde{f}(t)$. The
signal-to-noise ratio (SNR) of this recovery is $20\log_{10}(\|f\|/\|f-\tilde{f}\|)=-0.44$dB.

\begin{figure*}\centering
\includegraphics[width=\textwidth, trim=0cm 3cm 0cm 0cm ]{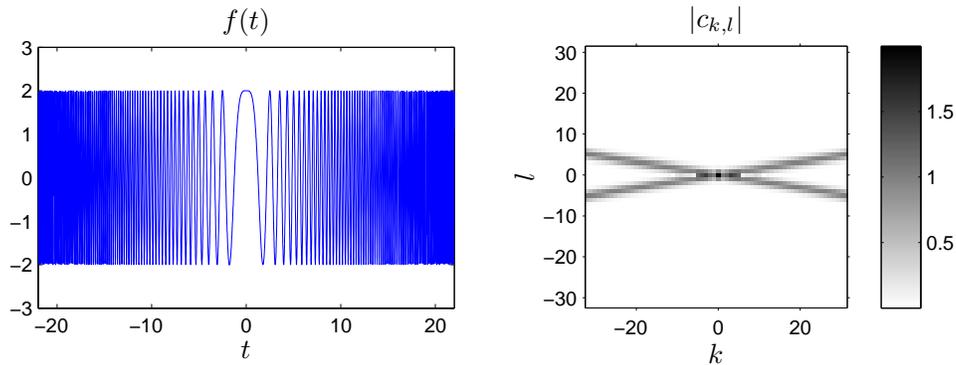}
\caption{A chirp signal and its Gaussian-window Gabor representation.} \label{fig:f_and_c}
\end{figure*}

The reconstructions obtained with the consistency and minimax-regret methods are shown in Fig.~\ref{fig:ChirpRec}(b)
and (c). Clearly, they both bear better resemblance to $f(t)$. The consistent recovery is the unique signal that can
be constructed with the synthesis window $v(t)$, whose Gabor transform coincides with $c_{k,l}$. This property makes
this reconstruction desirable in some sense, although its SNR is only $-1.03$dB, worse than the uncompensated
recovery. To guarantee that the error between our recovery $\tilde{f}(t)$ and the original signal $f(t)$ is small,
for every possible $f(t)$ that could have generated $c_{k,l}$, one has to use the minimax regret approach, as shown
Fig.~\ref{fig:ChirpRec}(c). This reconstruction achieves an SNR of $0.1$dB, and thus is better than the other two
methods in terms of reconstruction error. Figure~\ref{fig:d_con_mx} depicts the expansion coefficients $d_{k,l}$
corresponding to the two methods.

\begin{figure*}\centering
\subfloat[]{\includegraphics[scale=0.38, trim=0cm 0cm 0cm 0cm ]{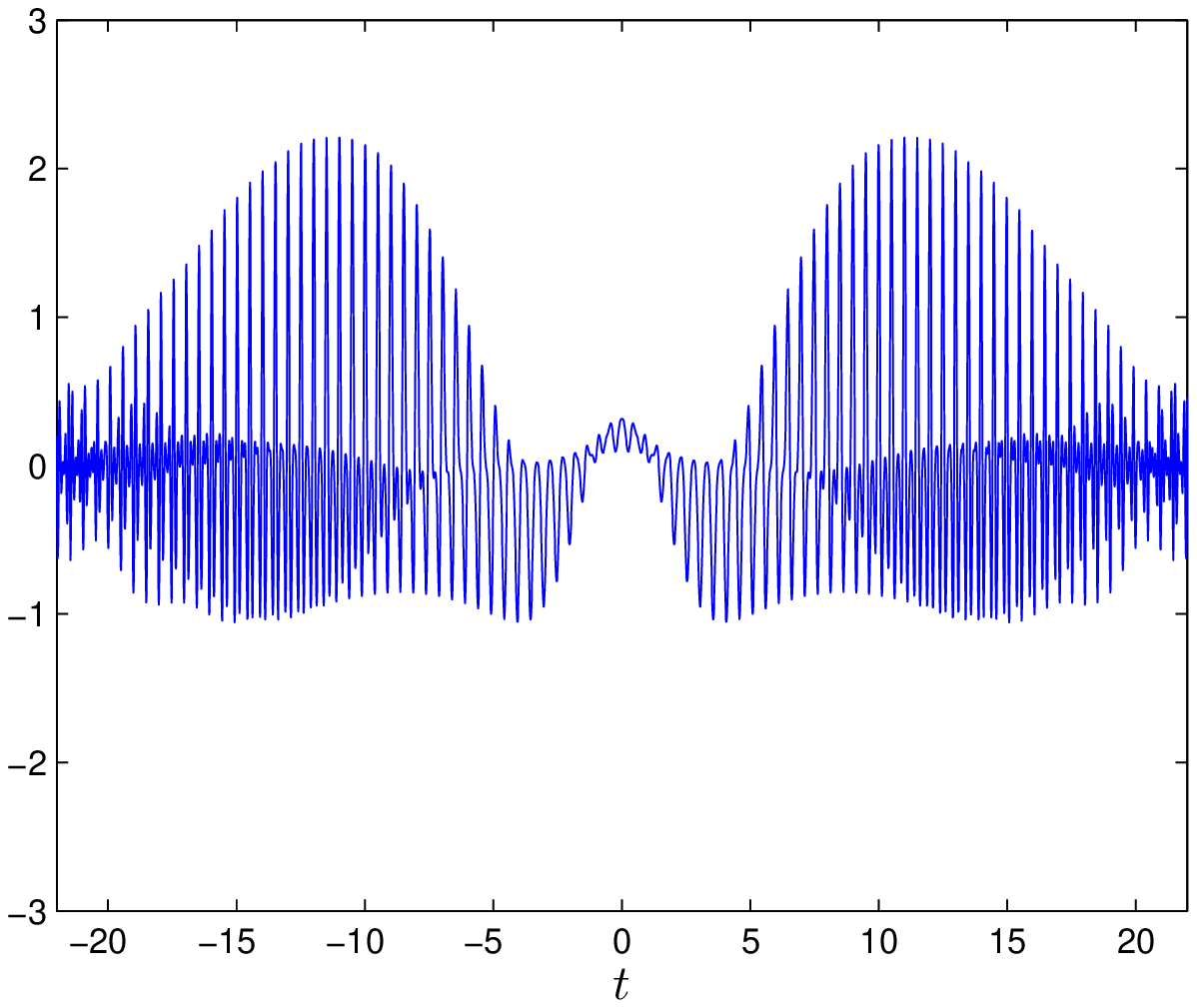}}
\subfloat[]{\includegraphics[scale=0.38, trim=0cm 0cm 0cm 0cm ]{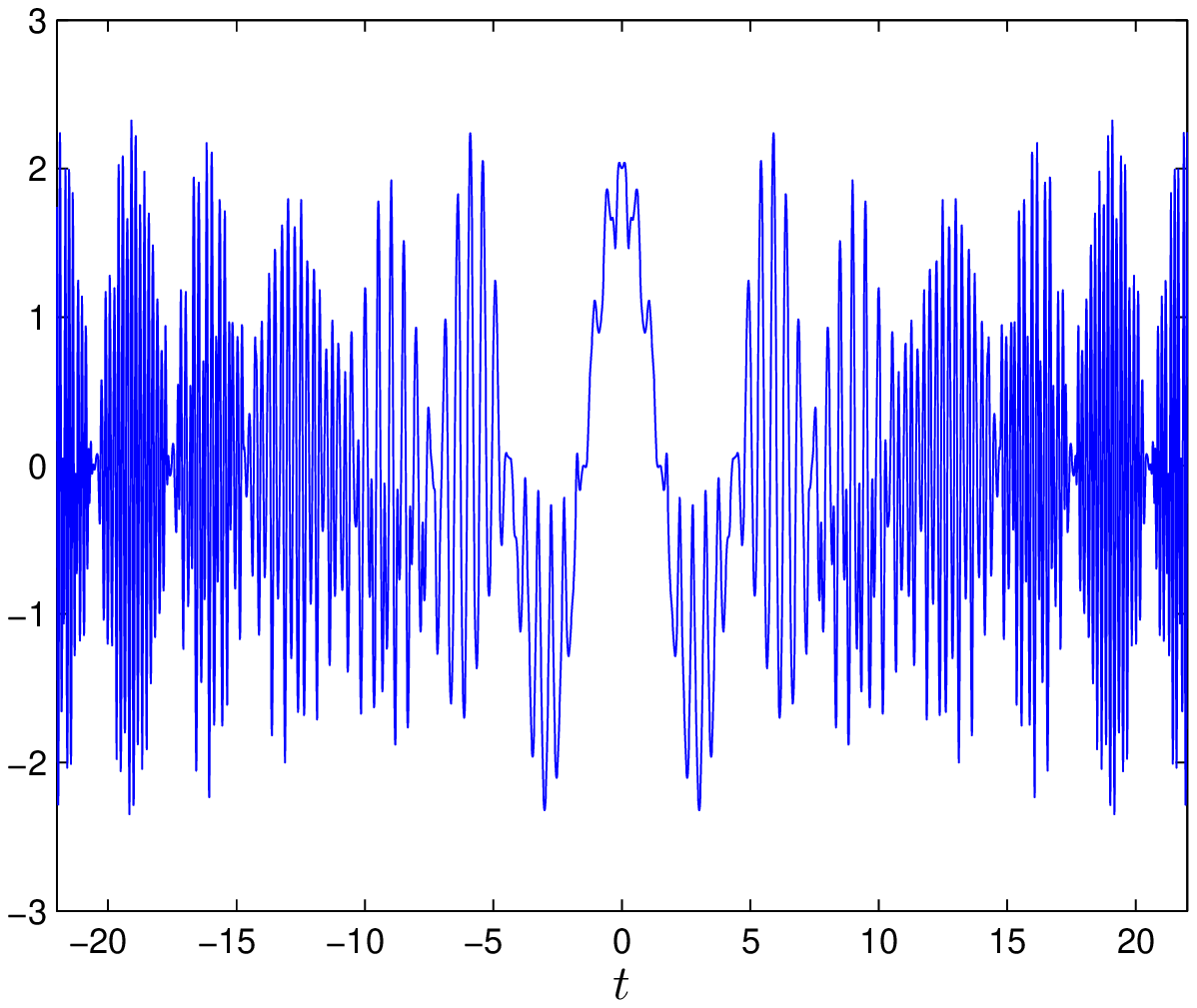}}
\subfloat[]{\includegraphics[scale=0.38, trim=0cm 0cm 0cm 0cm ]{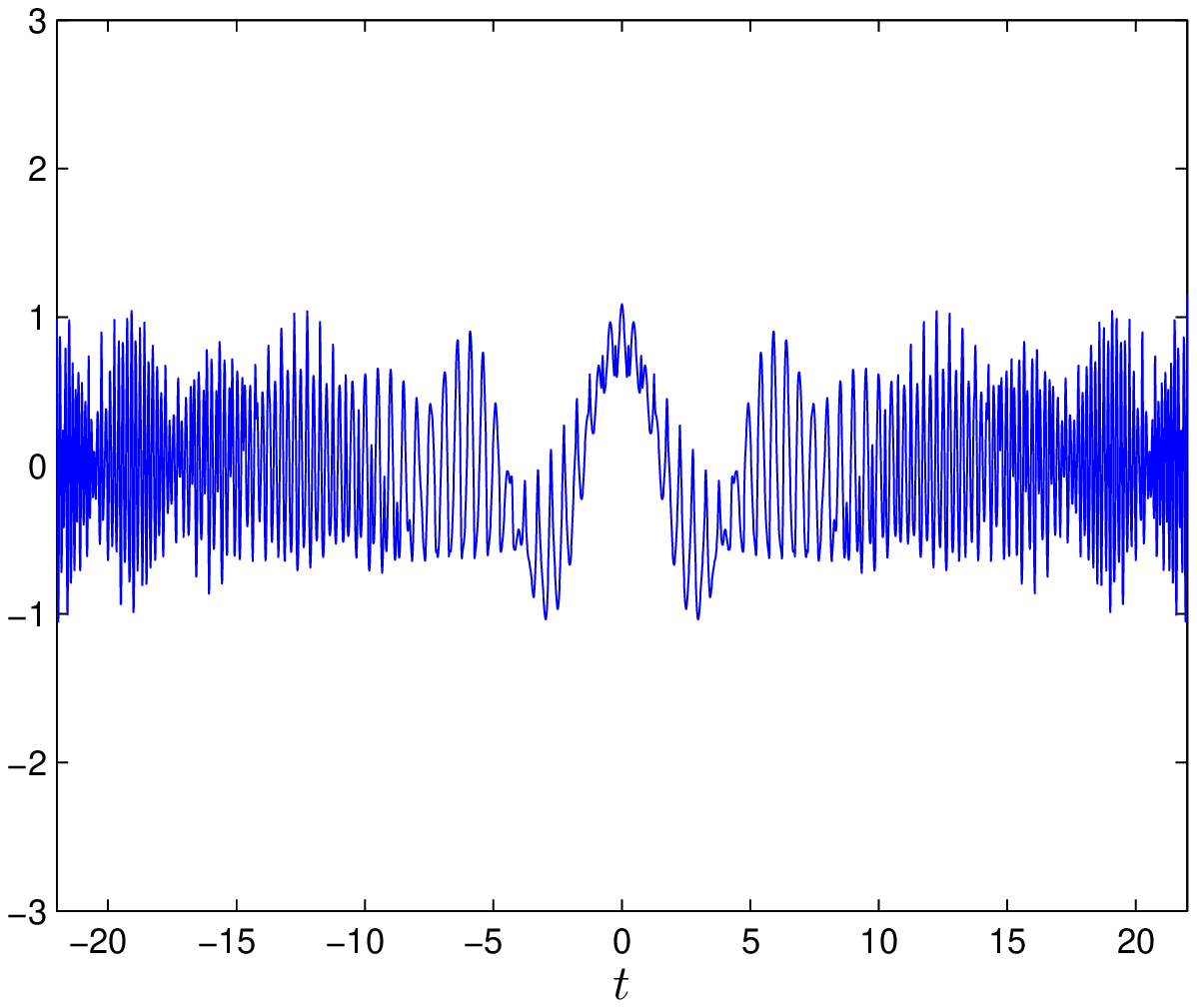}} \caption{Reconstructions of $f(t)$ from
its Gabor coefficients $c_{k,l}$. (a)~Without processing $c_{k,l}$. (b)~Consistent recovery, namely using
$d_{k,l}=(c*h_{\text{con}})_{k,l}$ as expansion coefficients. (c)~Minimax-regret recovery, namely using
$d_{k,l}=(c*h_{\text{mx}})_{k,l}$ as expansion coefficients.} \label{fig:ChirpRec}
\end{figure*}

\begin{figure*}\centering
\includegraphics[width=\textwidth, trim=0cm 3cm 0cm 0cm ]{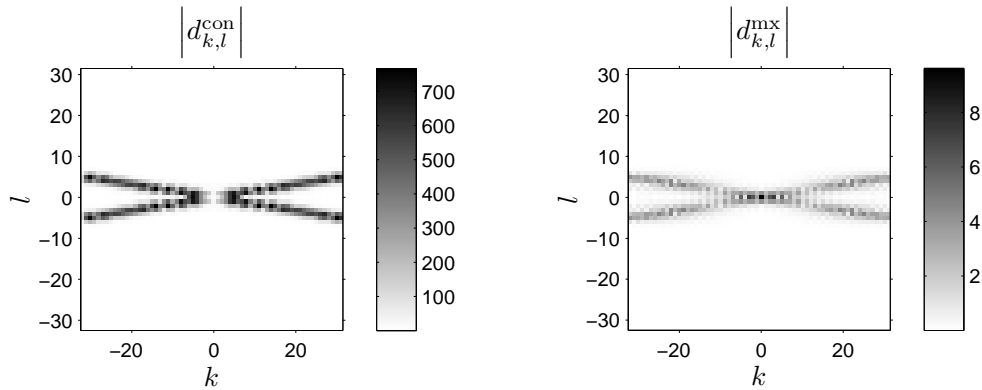}
\caption{The modulus of the expansion coefficients, $|d_{k,l}|$, corresponding to the consistent and minimax-regret
recovery methods.} \label{fig:d_con_mx}
\end{figure*}


\section{Conclusions}

In this paper we explored various techniques for recovering a signal from its non-invertible Gabor transform, where
the under-sampling factor is rational. Specifically, we studied situations where both the analysis and synthesis
windows of the transform are given, so that the only freedom is in processing the coefficients in the time-frequency
domain prior to synthesis. We began with the consistency approach, in which the recovered signal is required to
possess the same Gabor transform as the original signal. We then analyzed a minimax strategy whereby a reconstruction
with minimal worst case error is sought. Finally, we developed a recovery method yielding the minimal possible error
when the original signal is known to lie in some given Gabor space. We showed that all three techniques amount to
performing a 2D twisted convolution operation on the Gabor coefficients prior to synthesis. When the under-sampling
factor of the transform is an integer, this process reduces to standard convolution. We demonstrated our techniques
for Gaussian-window transforms in the context of recovering a chirp signal.

\appendices


\section{The Multiplication Property of the $\bPhi$ representation}
\label{sec:appPhiV}

Let $c_{k,l}$ and $d_{k,l}$ be two sequences having $\bPhi$ matrix-valued function representations $\bPhi^c$ and
$\bPhi^d$ respectively. Then the matrix-valued function $\bPhi$ associated with the twisted convolution $\twc{c}{d}$,
can be expressed as
\begin{equation}
\bPhi^{(\twc{c}{d})}(\om,x) = \bPhi^d(\om,x)\bPhi^c(\om,x).
\end{equation}
Indeed, let again $ab=q/p$ and let $r,s=0,\ldots,p-1$ be fixed, then
\begin{align*}
\bPhi^{(\twc{c}{d})}_{r,s}(\om,x) &= \sum_{k,l\in\Z} (\twc{c}{d})[s-r+pk,l]e^{-2\pi i abrl} e^{-2\pi i (blx+ak\om)}\\
&= \sum_{k,l\in\Z} \sum_{m,n\in\Z} c_{m,n}\,d_{s-r+pk-m,l-n}e^{-2\pi i ab(s-r+pk-m)n} e^{-2\pi i abrl} e^{-2\pi i (blx+ak\om)}\\
&= \sum_{u=0}^{p-1} \sum_{k,l\in\Z} \sum_{m,n\in\Z} c_{u+pm,n}d_{s-r-u+p(k-m),l-n} e^{-2\pi i ab(s-r-u)n} e^{-2\pi i abrl}e^{-2\pi i (blx+ak\om)}\\
&= \sum_{u=0}^{p-1} \sum_{k,l\in\Z}\sum_{m,n\in\Z} c_{s-u+pm,n}d_{u-r+pk,l} e^{-2\pi i ab(u-r)n} e^{-2\pi i abr(l+n)}e^{-2\pi i (b(l+n)x+a(k+m)\om)}\\
&= \sum_{u=0}^{p-1} \left(\sum_{k,l\in\Z} d_{u-r+pk,l} e^{-2\pi i abrl}e^{-2\pi i (blx+ak\om)} \right)
\left(\sum_{m,n\in\Z}c_{s-u+pm,n} e^{-2\pi i abun} e^{-2\pi i (bnx+am\om)} \right)\\
&= \sum_{u=0}^{p-1} \bPhi^d_{r,u}(\om,x) \bPhi^c_{u,s}(\om,x).
\end{align*}
Hence, $\bPhi^{(\twc{c}{d})}(\om,x) = \bPhi^d(\om,x)\bPhi^c(\om,x)$.


\section{Proof of Proposition~\ref{prop:equivalent basis}}
\label{sec:appProofEquivBasis}

Since $\G(v,a,b)$ is a Riesz basis for $\V$, there exist bounds $A>0$ and $B<\infty$ such that $A\bI_p \leq
\bPhi^{vv}(\om,x) \leq B\bI_p$, where $\bPhi^{vv}(\om,x)$ is the matrix-valued function associated to the sequence
$r_{vv}[k,l]$, defined in (\ref{eq:Phi-def}). The system $\G(w,a,b)$, with $w(t)=\sum_{k,l\in\Z}
h_{k,l}M_{bl}T_{ak}v(t)$, is a Riesz basis if and only if there exist constants $C>0$ and $D<\infty$ such that
\begin{equation}
C\bI_p \leq \bPhi^{ww}(\om,x) \leq D\bI_p,
\end{equation}
where $\bPhi^w(\om,x)$ is a matrix-valued function built from the cross-correlation sequence $r_{ww}[k,l] =
\inner{w}{M_{bl}T_{ak}w}$. By substituting $w(t)=\sum_{k,l\in\Z} h_{k,l}M_{bl}T_{ak}v(t)$ in $r_{ww}[k,l]$ one
obtains
\begin{align}
r_{ww}[k,l] &=  \sum_{y,z\in\Z} \sum_{m,n\in\Z} r_{vv}[y-m,z-n] h_{m,n} e^{-2\pi i ab(z-n)m} \overline{h_{y-k,z-l}}
e^{2\pi i ab(z-l)k}\nonumber\\
&=  \sum_{y,z\in\Z} (\twc{r_{vv}}{h})[y,z] \overline{h_{y-k,z-l}} e^{2\pi i ab(z-l)k}\nonumber\\
&= (\twc{h^{\ast}}{\twc{r_{vv}}{h}})[k,l]\,,
\end{align}
where $r_{vv}[m,n]=\inner{v}{M_{bn}T_{am}v}$ and $h^{\ast}[k,l] = \overline{h_{-k,-l}}$. It is easy to check, and
we leave it for the reader, that $\bPhi^{h^{\ast}}(\om,x) = \bPhi^h(\om,x)^H$. Therefore, using the relation
from Appendix~\ref{sec:appPhiV}, the $(r,s)$-entry of the matrix $\bPhi^{ww}(\om,x)$ is
\begin{equation}
\bPhi^{ww}_{r,s}(\om,x) = \left(\bPhi^h(\om,x) \bPhi^{vv}(\om,x) \bPhi^h(\om,x)^{H} \right)_{r,s}\,,
\end{equation}
where $\bPhi^h(\om,x)$ is a matrix-valued function associated to the sequence $h_{k,l}$ and defined in the
Proposition. Hence, if $\G(w,a,b)$ and $\G(v,a,b)$ are Riesz bases with bounds $C>0$, $D<\infty$, and $A>0$, $B<\infty$
respectively then
\begin{align}
\bPhi^h(\om,x) \bPhi^h(\om,x)^{H} &\geq \frac{1}{A} \bPhi^h(\om,x) \bPhi^{vv}(\om,x) \bPhi^h(\om,x)^{H} =
\frac{1}{A} \bPhi^{ww}(\om,x) \geq \frac{D}{A}\\
\bPhi^h(\om,x) \bPhi^h(\om,x)^{H} &\leq \frac{1}{B} \bPhi^h(\om,x) \bPhi^{vv}(\om,x) \bPhi^h(\om,x)^{H} = \frac{1}{B}
\bPhi^{ww}(\om,x) \leq \frac{C}{B}.
\end{align}
Therefore $\bPhi^h(\om,x)$ satisfies (\ref{eq:PhiPhi*}) with bounds $m=C/B$ and $M=D/A$.

On the other hand, if the sequence $h_{k,l}$ is such that (\ref{eq:PhiPhi*}) is satisfied, then
\begin{align}
\bPhi^{ww}(\om,x) &= \bPhi^h(\om,x) \bPhi^{vv}(\om,x) \bPhi^h(\om,x)^{H}
\geq A \bPhi^h(\om,x) \bPhi^h(\om,x)^{H} \geq Am \\
\bPhi^{ww}(\om,x)&= \bPhi^h(\om,x) \bPhi^{vv}(\om,x) \bPhi^h(\om,x)^{H} \leq B \bPhi^h(\om,x) \bPhi^h(\om,x)^{H} \leq
BM,
\end{align}
and so $\G(w,a,b)$ is a Riesz basis with bounds $C=Am$ and $D=BM$. It remains to show that $\G(w,a,b)$ and
$\G(v,a,b)$ span the same space. Every element of $\G(w,a,b)$ can be uniquely represented by a linear combinations of
the elements from $\G(v,a,b)$, since the latter is a Riesz basis. It suffices to show that $v(t)$ can be written as a
linear combination of the elements from $\G(w,a,b)$ (it will be a unique representation since $\G(w,a,b)$ is a Riesz
basis). Then, since Gabor spaces are closed under translation and modulations, other basis elements from $\G(v,a,b)$
will also admit a unique representation in terms of $\G(w,a,b)$. Let $g_{k,l}$ be the inverse of $h_{k,l}$ with
respect to $\nat$, meaning $\twc{h}{g}=\de$. The inverse exists because $h_{k,l}$ satisfies (\ref{eq:PhiPhi*}). Let
$\tilde{v}(t)=\sum_{m,n\in\Z} g_{m,n}M_{bn}T_{am}w(t)$. We will now show that $\tilde{v}(t)=v(t)$. Indeed,
\begin{align}
\tilde{v}(t) &= \sum_{m,n\in\Z} g_{m,n}M_{bn}T_{am}w(t) = \sum_{m,n\in\Z} \sum_{k,l\in\Z} g_{m,n}h_{k,l}M_{bn}T_{am}
M_{bl}T_{ak}v(t)\nonumber\\
&= \sum_{m,n\in\Z} g_{m,n} \sum_{k,l\in\Z} h_{k,l} e^{-2\pi i abml}M_{b(n+l)}T_{a(m+k)}v(t) \nonumber\\
&= \sum_{m,n\in\Z} \sum_{k,l\in\Z} g_{m-k,n-l}h_{k,l}e^{-2\pi i ab(m-k)l}M_{bn}T_{am}v(t) \nonumber\\
&= \sum_{m,n\in\Z} (\twc{h}{g})[m,n]M_{bn}T_{am}v(t) = v(t).
\end{align}

\bibliographystyle{IEEEtran}
\bibliography{Gabor}

\begin{thebibliography}{10}
\providecommand{\url}[1]{#1}
\csname url@samestyle\endcsname
\providecommand{\newblock}{\relax}
\providecommand{\bibinfo}[2]{#2}
\providecommand{\BIBentrySTDinterwordspacing}{\spaceskip=0pt\relax}
\providecommand{\BIBentryALTinterwordstretchfactor}{4}
\providecommand{\BIBentryALTinterwordspacing}{\spaceskip=\fontdimen2\font plus
\BIBentryALTinterwordstretchfactor\fontdimen3\font minus
  \fontdimen4\font\relax}
\providecommand{\BIBforeignlanguage}[2]{{%
\expandafter\ifx\csname l@#1\endcsname\relax
\typeout{** WARNING: IEEEtran.bst: No hyphenation pattern has been}%
\typeout{** loaded for the language `#1'. Using the pattern for}%
\typeout{** the default language instead.}%
\else
\language=\csname l@#1\endcsname
\fi
#2}}
\providecommand{\BIBdecl}{\relax}
\BIBdecl

\bibitem{EM84}
Y.~Ephraim and D.~Malah, ``{Speech enhancement using a minimum-mean square
  error short-time spectral amplitude estimator},'' \emph{IEEE Transactions on
  Acoustics, Speech and Signal Processing}, vol.~32, no.~6, pp. 1109--1121,
  1984.

\bibitem{CB01}
I.~Cohen and B.~Berdugo, ``{Speech enhancement for non-stationary noise
  environments},'' \emph{Signal Processing}, vol.~81, no.~11, pp. 2403--2418,
  2001.

\bibitem{BA98}
A.~Belouchrani and M.~G. Amin, ``{Blind source separation based on
  time-frequency signalrepresentations},'' \emph{IEEE Transactions on Signal
  Processing}, vol.~46, no.~11, pp. 2888--2897, 1998.

\bibitem{LM99}
Y.~Lu and J.~M. Morris, ``{Gabor expansion for adaptive echo cancellation},''
  \emph{IEEE signal processing magazine}, vol.~16, no.~2, pp. 68--80, 1999.

\bibitem{ACV01}
C.~Avendano, C.~A.~T. Center, and S.~Valley, ``{Acoustic echo suppression in
  the STFT domain},'' in \emph{Proc. IEEE Workshop Applicat. Signal Process.
  Audio Acoust.}, 2001, pp. 175--178.

\bibitem{AG99}
C.~Avendano and G.~Garcia, ``{STFT-based multi-channel acoustic interference
  suppressor},'' in \emph{Proc. Int. Conf. Acoust., Speech, Signal Processing
  (ICASSP'01)}, vol.~1, 2001.

\bibitem{C04}
I.~Cohen, ``{Relative transfer function identification using speech signals},''
  \emph{IEEE Transactions on Speech and Audio Processing}, vol.~12, no.~5, pp.
  451--459, 2004.

\bibitem{GBW01}
S.~Gannot, D.~Burshtein, and E.~Weinstein, ``{Signal enhancement using
  beamforming and nonstationarity withapplications to speech},'' \emph{IEEE
  Transactions on Signal Processing}, vol.~49, no.~8, pp. 1614--1626, 2001.

\bibitem{AC07}
Y.~Avargel and I.~Cohen, ``{System identification in the short-time Fourier
  transform domain with crossband filtering},'' \emph{IEEE Transactions on
  Audio Speech and Language Processing}, vol.~15, no.~4, p. 1305, 2007.

\bibitem{AC08}
------, ``{Adaptive system identification in the short-time Fourier transform
  domain using cross-multiplicative transfer function approximation},''
  \emph{IEEE Transactions on Audio Speech and Language Processing}, vol.~16,
  no.~1, p. 162, 2008.

\bibitem{TZ96}
S.~Tomazic and S.~Znidar, ``{A fast recursive STFT algorithm},'' in \emph{8th
  Mediterranean Electrotechnical Conference, 1996. MELECON'96.}, vol.~2, 1996.

\bibitem{U00}
M.~Unser, ``Sampling---50 years after {S}hannon,'' \emph{IEEE Proc.}, vol.~88,
  pp. 569--587, Apr. 2000.

\bibitem{EM09}
Y.~C. Eldar and T.~Michaeli, ``Beyond bandlimited sampling,'' \emph{IEEE signal
  processing magazine}, vol.~26, no.~3, pp. 46--68, May 2009.

\bibitem{UA94}
M.~Unser and A.~Aldroubi, ``A general sampling theory for nonideal acquisition
  devices,'' \emph{IEEE Trans. Signal Process.}, vol.~42, no.~11, pp.
  2915--2925, Nov. 1994.

\bibitem{ED04}
Y.~C. Eldar and T.~G. Dvorkind, ``A minimum squared-error framework for
  generalized sampling,'' \emph{IEEE Trans. Signal Process.}, vol.~54, no.~6,
  pp. 2155--2167, Jun. 2006.

\bibitem{UAE92}
M.~Unser, A.~Aldroubi, and M.~Eden, ``{Polynomial spline signal approximations:
  filter design andasymptotic equivalence with Shannon's sampling theorem},''
  \emph{IEEE Transactions on Information Theory}, vol.~38, no.~1, pp. 95--103,
  1992.

\bibitem{UAE01}
------, ``{B-S}pline signal processing: Part {I} - {T}heory,'' \emph{IEEE
  Trans. Signal Process.}, vol.~41, no.~2, pp. 821--833, Feb 1993.

\bibitem{AG01}
A.~Aldroubi and K.~Gr{\"o}chenig, ``Non-uniform sampling and reconstruction in
  shift-invariant spaces,'' \emph{{\em Siam Review}}, vol.~43, pp. 585--620,
  2001.

\bibitem{A02}
A.~Aldroubi, ``Non-uniform weighted average sampling and exact reconstruction
  in shift-invariant and wavelet spaces,'' \emph{Appl. Comp. Harmonic. Anal.},
  vol.~13, pp. 151--161, 2002.

\bibitem{CE04}
O.~Christansen and Y.~C. Eldar, ``Oblique dual frames and shift-invariant
  spaces,'' \emph{Applied and Computational Harmonic Analysis}, vol.~17, no.~1,
  2004.

\bibitem{UB05}
M.~Unser and T.~Blu, ``Cardinal exponential splines: part {I} - theory and
  filtering algorithms,'' \emph{IEEE Trans. Signal Processing}, vol.~53, no.~4,
  pp. 1425 -- 1438, Apr. 2005.

\bibitem{UN05}
------, ``{Generalized smoothing splines and the optimal discretization of the
  Wiener filter},'' \emph{IEEE Trans. Signal Process.}, vol.~53, no.~6, pp.
  2146--2159, 2005.

\bibitem{EU06}
Y.~C. Eldar and M.~Unser, ``Nonideal sampling and interpolation from noisy
  observations in shift-invariant spaces,'' \emph{IEEE Trans. Signal Process.},
  vol.~54, no.~7, pp. 2636--2651, Jul. 2006.

\bibitem{RA08}
S.~Ramani, D.~Van De~Ville, T.~Blu, and M.~Unser, ``{Nonideal Sampling and
  Regularization Theory},'' \emph{IEEE Trans. Signal Process.}, vol.~56, no.~3,
  pp. 1055--1070, 2008.

\bibitem{MI08}
T.~Michaeli and Y.~C. Eldar, ``{High Rate Interpolation of Random Signals from
  Nonideal Samples},'' \emph{IEEE Trans. Signal Processing}, vol.~57, no.~3,
  2009.

\bibitem{RA06}
S.~Ramani, D.~Van De~Ville, and M.~Unser, ``{non-ideal sampling and adapted
  reconstruction using the stochastic Matern model},'' \emph{Proc. Int. Conf.
  Acoust., Speech, Signal Processing (ICASSP'06)}, vol.~2, 2006.

\bibitem{EMW07}
Y.~C. {E}ldar, E.~{M}atusiak, and T.~{W}erther, ``{A} {C}onstructive inversion
  framework for twisted convolution,'' \emph{Monatsh. Math.}, vol. 150, no.~4,
  2007.

\bibitem{G01}
K.~Gro{\"c}henig, \emph{Foundations of Time-Frequency Analysis}.\hskip 1em plus
  0.5em minus 0.4em\relax Birkh{\"a}user, Boston, 2001.

\bibitem{RS97}
A.~Ron and Z.~Shen, ``{Weyl-Heisenberg frames and Riesz bases in
  $L_2(\mathbb{R}^d$},'' \emph{Duke Math. J.}, vol.~89, no.~2, pp. 237--282,
  1997.

\bibitem{WES05}
T.~Werther, Y.~C. Eldar, and N.~K. Subbana, ``{Dual Gabor Frames: Theory and
  Computational Aspects},'' \emph{IEEE Trans. Signal Process.}, vol.~53,
  no.~11, pp. 4147--4158, 2005.

\bibitem{WMES07}
T.~Werther, E.~Matusiak, Y.~C. Eldar, and N.~K. Subbana, ``{A Unified approach
  to dual Gabor windows},'' \emph{IEEE Trans. Signal Process.}, vol.~55, no.~5,
  pp. 1758--1768, May 2007.

\bibitem{TE10}
T.~Michaeli and Y.~C. Eldar, ``Optimization techniques in modern sampling
  theory,'' in \emph{to appear in Convex Optimization in Signal Processing and
  Communications}, Y.~C. Eldar and P.~D., Eds.\hskip 1em plus 0.5em minus
  0.4em\relax Cambridge University Press.

\bibitem{EW05}
Y.~C. Eldar and T.~Werther, ``General framework for consistent sampling in
  {H}ilbert spaces,'' \emph{International Journal of Wavelets, Multiresolution,
  and Information Processing}, vol.~3, no.~3, pp. 347--359, Sep. 2005.

\bibitem{A96}
A.~Aldroubi, ``Oblique projections in atomic spaces,'' \emph{Proc. Amer. Math.
  Soc.}, vol. 124, no.~7, pp. 2051--2060, 1996.

\bibitem{E02}
Y.~C. Eldar, ``Sampling and reconstruction in arbitrary spaces and oblique dual
  frame vectors,'' \emph{J. Fourier Analys. Appl.}, vol.~1, no.~9, pp. 77--96,
  Jan. 2003.

\bibitem{E03}
------, ``Sampling and reconstruction in arbitrary spaces and oblique dual
  frame vectors,'' \emph{J. Fourier Analys. Appl.}, vol.~1, no.~9, pp. 77--96,
  Jan. 2003.

\bibitem{E04}
------, ``Sampling without input constraints: Consistent reconstruction in
  arbitrary spaces,'' in \emph{Sampling, Wavelets and Tomography}, A.~I. Zayed
  and J.~J. Benedetto, Eds.\hskip 1em plus 0.5em minus 0.4em\relax Boston, MA:
  Birkh{\"a}user, 2004, pp. 33--60.

\bibitem{B99}
D.~{B}ertsekas, \emph{{N}onlinear {P}rogramming}, 2nd~ed.

\bibitem{ENB03}
Y.~C. Eldar, A.~Ben-Tal, and A.~Nemirovski, ``Linear minimax regret estimation
  of deterministic parameters with bounded data uncertainties,'' \emph{IEEE
  Trans. Signal Process.}, vol.~52, pp. 2177--2188, Aug. 2004.

\end{thebibliography}

\end{document}